\documentclass[11pt,reqno]{amsart}

\usepackage{amsmath,amssymb,amsfonts}
\usepackage[left=30mm,right=30mm,top=30mm,bottom=30mm]{geometry}
\usepackage{hyperref}
\usepackage{cite}
\usepackage[foot]{amsaddr}
\usepackage[sc]{mathpazo}
\usepackage{graphics}
\usepackage{graphicx}
\usepackage{color}

\newcommand{\1}{\mathbf{1}}

\newcommand{\Z}{\mathbb{Z}}
\newcommand{\N}{\mathbb{N}}
\newcommand{\R}{\mathbb{R}}
\newcommand{\C}{\mathbb{C}}

\newcommand{\euler}{\mathrm e}  

\renewcommand{\L}{\mathsf{L}^2}
\renewcommand{\H}{\mathsf{H}}
\newcommand{\HS}{\mathcal{H}}

\renewcommand{\a}{\mathfrak{a}}
\newcommand{\A}{{\mathcal{A}}}
\newcommand{\J}{{\mathcal{J}}} 
\newcommand{\Res}{{\mathcal{R}}}

\DeclareMathOperator{\dom}{dom}

\newcommand{\la}{\langle}
\newcommand{\ra}{\rangle}

\newcommand{\eps}{\varepsilon}
\newcommand{\e}{_{\varepsilon}}
\newcommand{\al}{\alpha}
\newcommand{\be}{\beta}
\newcommand{\ga}{\gamma}

\newcommand{\x}{\mathbf{x}}

\renewcommand{\d}{\,\mathrm{d}}

\newcommand{\ds}{\displaystyle}

\DeclareMathOperator{\dist}{\mathrm{dist}}

\newcommand{\Id}{\mathrm{I}}

\newcommand{\wt}{\widetilde}
\DeclareMathOperator{\intr}   {int}  
\newcommand{\Hausdorff}{\mathrm H} 

\usepackage{mathtools}
\DeclarePairedDelimiter\abs{\lvert}{\rvert}%
\DeclarePairedDelimiter\norm{\lVert}{\rVert}%

\theoremstyle{plain}
\newtheorem{theorem}{Theorem}[section]
\newtheorem*{theorem*}{Theorem}
\newtheorem{lemma}[theorem]{Lemma}
\newtheorem*{lemma*}{Lemma}

\newtheorem{proposition}[theorem]{Proposition}
\theoremstyle{definition}
\newtheorem{remark}[theorem]{Remark}
\newtheorem*{remark*}{Remark} 
\newtheorem{definition}[theorem]{Definition}

\numberwithin{equation}{section}

\begin{document}
\title
[A geometric approximation of $\delta$-interactions by Neumann Laplacians]
{A geometric approximation of $\delta$-interactions by Neumann Laplacians}

\author[Andrii Khrabustovskyi]{Andrii Khrabustovskyi\,$^{1,2}$}
\address{$^1$ Department of Physics, Faculty of Science, University of
  Hradec Kr\'{a}lov\'{e}, Czech Republic} 
\address{$^2$ Department of Theoretical Physics,
Nuclear Physics Institute of the Czech Academy of Sciences, \v{R}e\v{z}, Czech Republic} 
\email{andrii.khrabustovskyi@uhk.cz}

\author[Olaf Post]{Olaf Post\,$^{3}$}
\address{$^3$ Department of Mathematics, University of Trier,
Germany}
\email{olaf.post@uni-trier.de}

\begin{abstract}
  We demonstrate how to approximate one-dimensional Schr\"odinger operators with $\delta$-interaction by a Neumann Laplacian on a narrow waveguide-like domain. Namely, we consider a domain consisting of a straight  strip and a small protuberance with ``room-and-passage'' geometry. We show that in the limit when the perpendicular size of the strip tends to zero, and the room and the passage are appropriately scaled, the Neumann Laplacian on this domain converges in (a kind of) norm resolvent sense to the above singular Schr\"odinger operator. 
Also we prove Hausdorff convergence of the spectra.
In both cases estimates on the rate of convergence are derived. 
\end{abstract}


\maketitle

\thispagestyle{empty}

\section{Introduction}
\label{sec1}

Schr\"odinger operators with  potentials supported on a discrete set of points 
have attracted considerable attention over several past decades due to numerous applications in different fields of  science  and  engineering. 
In particular, such operators serve as \emph{solvable models} in quantum mechanics. The term ``solvable''  reflects the fact that their mathematical and physical
quantities, like spectrum, eigenfunctions and resonances, can be calculated in many cases explicitly.
We refer to the monograph~\cite{AGHH05} for a comprehensive introduction to this topic. 
Note that in the literature such models are also called Schr\"odinger operators with \emph{point interactions}.
 
Investigation of these operators was originated by the famous \emph{Kronig-Penney model}~\cite{KP31},
concerning a non-relativistic electron moving in a fixed crystal lattice. 
Its mathematical representation is a one-dimensional Schr\"odinger operator with a singular potential supported at $\Z=\{0,\pm 1,\pm 2,\dots\}$:  
\begin{gather}\label{KP}
  -\frac{\d^2}{\d x^2} + \ga\sum_{z\in\Z}\delta(\cdot - z),
\end{gather}
where $\delta(\cdot - z)$ denotes the Dirac delta-function supported at $z$, and where $\gamma\in\R$ is the strength of the $\delta$-interaction.  
The formal expression~\eqref{KP} can be realized as a self-adjoint operator    in $\L(\R)$: the action 
of this operator is given by  $-(f\restriction_{\R\setminus \Z})''$, its domain consists of $f\in \mathsf{H}^2(\R\setminus \Z)$ satisfying the following   conditions at   $z\in\Z$:
\begin{gather}
  \label{delta-conditions}
  f(z-0) = f(z + 0),\quad f'(z+0)-f'( z-0)=\ga f( z\pm 0).
\end{gather}
One says that the conditions~\eqref{delta-conditions} correspond to a $\delta$-interaction with strength $\ga$ supported at the point $z$.

In the present paper
we wish to contribute to the understanding of ways 
how Schr\"odinger operators with $\delta$-interactions can be approximated by differential operators with regular coefficients. 
In what follows, we restrict ourselves to a Schr\"odinger operator on an open interval $\Omega_0$ (bounded or not) with a single $\delta$-interaction with strength $\gamma$ supported at $0$; we denote this operator by $\A_0$.
The obtained results can easily be extended to Schr\"odinger operators with finitely and even countably many $\delta$-interactions, see Section~\ref{sec5}.  

One way of approximation is given by a sequence of Schr\"odinger operators with smooth $\delta$-like potentials, see e.g.~\cite[Sec.~1.3.2]{AGHH05}.
In the present work we discuss another approach, in which the desired $\delta$-interaction
is generated by geometry. 
Namely, we  construct a waveguide-like domain $\Omega\e$ such that its Neumann Laplacian $\A\e$ converges (in a kind of norm-resolvent topology) to the desired operator $\A_0$ as the perpendicular size of the guide tends to zero.  Since the approximating operator is non-negative (minus the Laplacian), we can only expect $\gamma \ge 0$ in the limit (in our convention, $\gamma \ge 0$ leads to a non-negative operator with $\delta$-interaction, see~\eqref{ae0} for the corresponding form).

\begin{figure}[h]
  \begin{center}
    \begin{picture}(290,90)
      \includegraphics[width=0.6\textwidth]{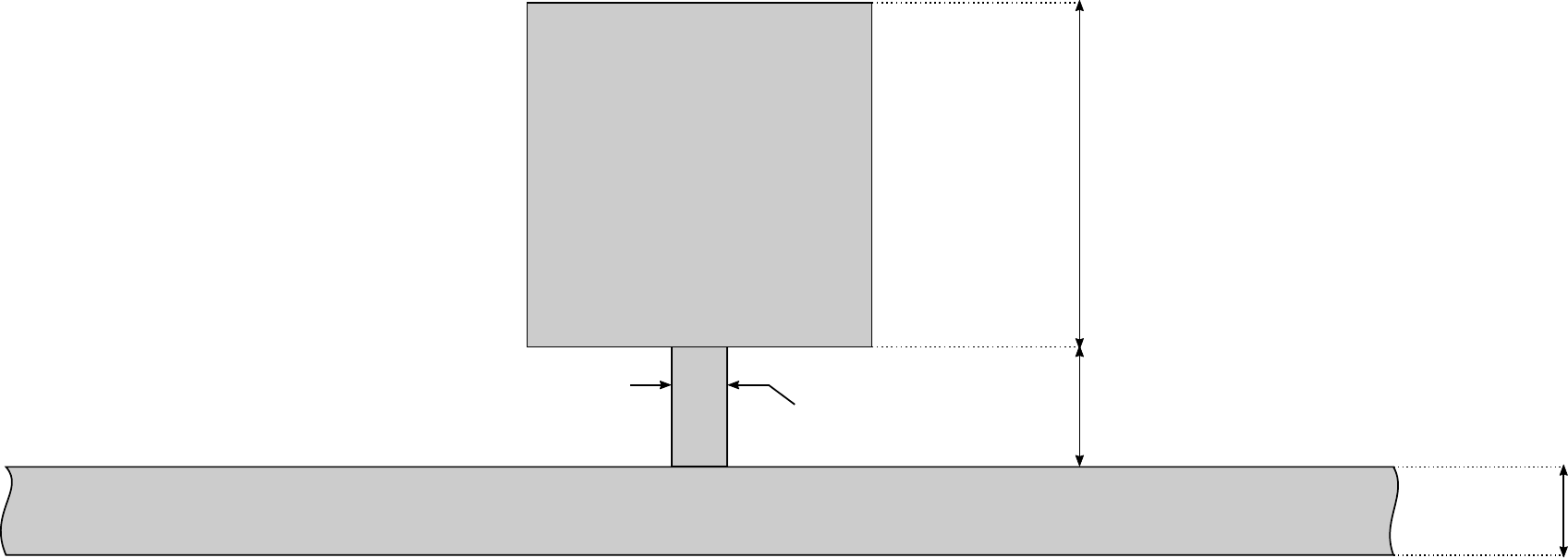}
      \put(-120,5){$S\e$}
      \put(-160,60){$R\e$}
      \put(-152,20){$P\e$}
      \put(1,5){$\eps$}
      \put(-130,23){$d\e$}
      \put(-81,23){$h\e$}
      \put(-81,63){$b\e$}
    \end{picture}
  \end{center}
  \caption{The waveguide $\Omega\e$}
  \label{fig1}
\end{figure}
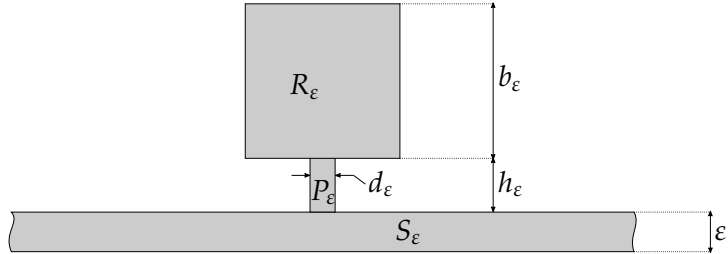

The approximating domain will be a thickened version of $\Omega_0$
 with a decoration near $0 \in \Omega_0$ given by a small passage
 $P_\eps$ and a larger (but shrinking) room $R_\eps$ (see
 Figure~\ref{fig1}).  In particular, the room  (a square with
 side length $b_\eps=\eps^\beta$) can be chosen to shrink arbitrarily
 slow ($\beta \in (0,1/2)$), and the passage of height
 $h_\eps=\eps^\alpha$ and width $d_\eps=\gamma\eps^{\alpha+1}$ can
 shrink arbitrarily fast ($\alpha>0$).
 Nevertheless, the area $\eps^{2\beta}$ of the room $R_\eps$ is still
 shrinking compared to the transversal shrinking rate $\eps$ of the
 strip $S_\eps$.  The strength of the $\delta$-interaction is given by
 $\gamma=d_\eps/(h_\eps \eps)$.  Note that we can interpret the quotient of
 the width $d_\eps$ and the height $h_\eps$ of the passage as the
 (vertical) \emph{conductance} of $P_\eps$.

It is interesting to compare our results   with the 
the case, when the room  is joined
directly to the strip near $0$ as on Figure~\ref{fig2}. 
Such a geometrical configuration was considered in \cite{KuZ03,EP05}  (see also~\cite{P12} for a version with
 generalised norm resolvent convergence).
The critical value $1/2$ (in dimension $2$) of the parameter $\beta$ appears here as well.
Namely, if $\beta\in (0,1/2)$, then the limit operator is the
 direct sum of the Laplacian with Dirichlet boundary condition at $0$
 (hence decouples) and the $0$ operator on a one-dimensional space as
 in our case. 
 In both cases (Figure~\ref{fig1} and Figure~\ref{fig2}) the room area decays slower than the transversal shrinking rate, which attracts particles and leads to an own state (at energy $0$).  
 If the room is directly attached to the strip (Figure~\ref{fig2}), then it prevents transport along the strip, while in the present situation (Figure~\ref{fig1}) it leads to a repulsive (i.e.\ with $\gamma\ge 0$) interaction at $0$. 
Note that for the waveguide as on Figure~\ref{fig2} in the critical case $b\e=\eps^{1/2}$ 
the limiting operator is   
the Laplacian with another peculiar conditions at $0$ (see the so-called ``borderline case'' in~\cite{EP05} or\cite{P12}); these conditions resemble~\eqref{delta-conditions}
with the coupling constant $\gamma$ being replaced by a quantity dependent on the spectral parameter.

\begin{figure}[h]
  \begin{center}
      \includegraphics[width=0.5\textwidth]{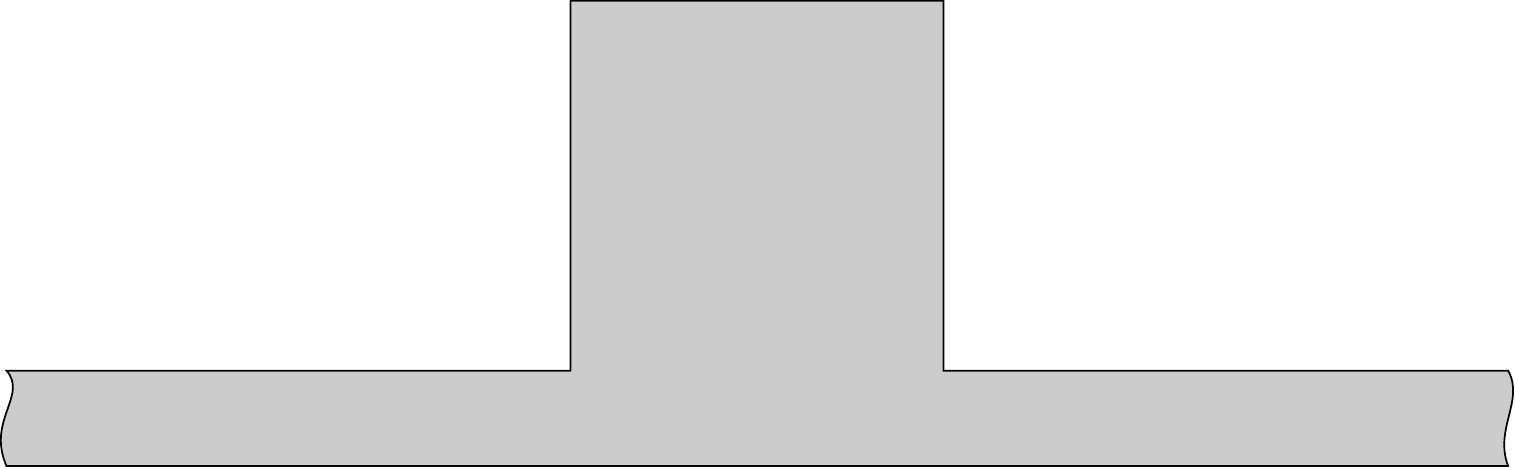}  
      \caption{  \label{fig2}}
  \end{center}    
\end{figure}

Domains with attached protuberances with ``room-and-passage'' geometry are widely used in spectral theory  in order to demonstrate various peculiar effects. For example, R.~Courant and D.~Hilbert~\cite{CH53} used such a domain as an example of a small  perturbation breaking the continuity of eigenvalues of the Neumann Laplacian; see~\cite{AHH91} for more details.
In~\cite{HSS91}  such ``rooms-and-passages'' were used to construct 
a domain such that its 
Neumann Laplacian has prescribed essential spectrum (see also the overview \cite{BK19} for more details). 
Homogenization problems in domains with corrugated ``room-and-passage''-like boundary were studied in~\cite{CK17,CK15}.
Finally, various peculiar examples in the theory of Sobolev spaces are based on domains with such a geometry,
see~\cite{Am78,Fr79,EH87} and the monograph~\cite{Ma11}.

As the spaces change while passing to the limit $\eps \to 0$, we use the framework of \emph{generalised norm resolvent convergence} developed by the second author in~\cite{P06} and~\cite{P12}.  We provide a self-contained presentation including a new proof of spectral convergence (cf.\ Theorem~\ref{thA:2}) in Section~\ref{sec3}.  As usual the generalised \emph{norm} resolvent convergence is not much harder to show than other concepts such as versions of \emph{strong} resolvent convergence used e.g.\ in homogenization theory~\cite[Chap.~III]{OSY92},~\cite[Chap.~XI]{ZKO94}.

Actually, our limit operator is \emph{not} the Laplacian with $\delta$-interaction itself (as already mentioned above), but its direct sum with the null operator on a one-dimensional space. 
In Remark~\ref{rem:deco} we give some light on the appearance of this extra component.

The work is organized as follows. In Section~\ref{sec2} we set the problem and formulate the main results, Theorem~\ref{th1} concerning the norm resolvent convergence and Theorem~\ref{th2} concerning the spectral convergence.
Note that we treat even more general operators 
$-\frac{\d^2}{\d x^2}+V_0+\delta(\cdot)$, where $V_0$ is a regular potential.
In Section~\ref{sec3} we give two abstract results designed for studying convergence of operators in varying Hilbert spaces. Using them we prove the main results in Section~\ref{sec4}. 
Finally, in Section~\ref{sec5} we discuss the case of countably many point interactions.

\section{Setting of the problem and the main result}
\label{sec2}

\subsection{The waveguide $\Omega\e$ and the operator $\A\e$}

Throughout the paper we denote points in $\R^2$ by $\x=(x_1,x_2)$.
Let $\Omega_0=(\ell_-,\ell_+)\subset\R$ be an interval satisfying  
\begin{gather}\label{ab}
  -\infty\le \ell_-<0< \ell_+\le\infty.
\end{gather}
Let $\eps\in (0,\eps_0]$  be a small parameter. We set
\begin{gather}\label{dhb}
  d\e=\gamma\eps^{\al+1},\quad h\e=\eps^\al,\quad 
  b\e=\eps^\be
  \text{\quad with\quad }\al>0,\ 0<\be<\frac12,\ \gamma>0.
\end{gather}
Moreover, we claim $\eps_0$ to be sufficiently small, namely
\begin{equation}
\label{eq:def.eps0}
\eps_0 < \min\Bigl\{(2\gamma)^{-1/\al},|\ell_-|,\ell_+\Bigl\},\quad
\eps_0<\gamma^{-1/(\al+1-\be)},\quad\text{and}\quad
\eps_0<\min\Bigl\{\ga,\ga^{-1}\Bigl\}.
\end{equation}
The first assumption in~\eqref{eq:def.eps0} implies
\begin{gather}
  \label{eps:small:1}
  d\e
  \leq \frac\eps2
  \qquad\text{and}\qquad
  [-\eps,\eps] \subset\Omega_0,
\end{gather}
the second one leads to
\begin{gather}\label{eps:small:2}
d\e\leq b\e,
\end{gather}
and the last one yields $|\ln\gamma|\leq |\ln\eps|$ (it will be used in the proof of Lemma~\ref{lm:aux}).
Note that, since either $\gamma\le 1$ or $\gamma^{-1}\le 1$, one has $\eps_0<1$. 
Finally, we introduce the  domains
\begin{equation*}
  \begin{array}{llll}
    S\e&=&\left\{\x\in\R^2:\ x_1\in \Omega_0,\ x_2\in (-\eps,0)\right\}&\text{(straight strip)},\\[2mm]
    P\e&=&\left\{\x\in\R^2:\  |x_1|<{\dfrac{d\e} 2},\ x_2\in (0,h\e)\right\}&\text{(passage)},\\[2mm]
    R\e&=&\left\{\x\in\R^2:\ |x_1|< \dfrac{b\e}2,\ x_2\in (h\e,h\e+b\e)\right\}&\text{(room)},
\end{array}
\end{equation*}
and the resulting domain $\Omega\e$  given by  
\begin{equation*}
  \Omega\e=\intr(\overline{S\e\cup P\e \cup R\e})
\end{equation*}
(here $\intr(\cdot)$ stands for the interior of a subset of $\R^2$).
Due to~\eqref{eps:small:1}--\eqref{eps:small:2} the geometry of $\Omega\e$ is exactly as 
shown in Figure~\ref{fig1}, i.e.\ the bottom part (respectively, the top part) of $\partial P\e$
is contained in the top part of  $\partial S\e$ (respectively, the bottom part of  $\partial R\e$).

In the Hilbert space $\HS\e\coloneqq\L(\Omega\e)$ we introduce the sesquilinear form 
\begin{gather}
\label{ae}
\a\e[u,v]=
\int_{\Omega\e}
\left(\nabla u(\x)\cdot \overline{\nabla v(\x)}  + V\e(\x)u(\x)\overline{v(\x)}\right)\d\x,\quad
\dom(\a\e)=\H^1(\Omega\e)
\end{gather}
with a real-valued potential $V\e\in\mathsf{L}^\infty(\Omega\e)$.
This form is densely defined in $\L(\Omega\e)$, non-negative and  closed, consequently
\cite[Theorem~VI.2.1]{Ka66} there is a unique non-negative 
self-adjoint operator $\A\e$ acting in $\L(\Omega\e)$ such that the domain inclusion 
$\dom(\A\e)\subset\dom(\a\e)$ and the equality
\begin{equation*}
(\A\e u, v)_{\L(\Omega\e)}=\a\e [u,v],\quad \forall u\in \dom(\A\e),\,\,v\in\dom(\a\e)
\end{equation*}
hold. Obviously, $\A\e=-\Delta_{\Omega\e}+V\e$, where $\Delta_{\Omega\e}$ is the Neumann Laplacian on $\Omega\e$.\smallskip

The main goal of this work is to describe the behavior of the resolvent and the spectrum of $\A\e$ as $\eps\to 0$. In the next subsection we introduce the anticipated limiting operator.

\subsection{The operator $\A_0$}

Recall that $\Omega_0\subset\R$ is an open interval containing $0$, see~\eqref{ab}.
We denote \begin{equation*}\HS_0\coloneqq \L(\Omega_0)\oplus \C,\end{equation*} i.e.\ $\HS_0$ is a Hilbert space consisting of   $f=(f_1,f_2)\in \L(\Omega_0)\times \C$
equipped with the scalar product
\begin{equation*}
(f,g)_{\HS_0}=\int_{\Omega_0} f_1 \overline{g_1}\d x + f_2\overline{g_2},\quad
f=(f_1,f_2),\ g=(g_1,g_2).
\end{equation*}
In the space $\HS_0$ we introduce the sesquilinear form $\a_0$   defined by
\begin{gather}
\label{ae0}
\begin{array}{c}
\a_0[f,g]=\ds\int_{\Omega_0}\left( f_1'(x)\overline{g_1'(x)}+V_0(x)f_1(x)\overline{g_1(x)} \right)\d x+
\gamma\, f_1(0)\overline{g_1(0)},\\[2mm]
\dom(\a_0)=\H^1(\Omega_0)\times\C,
\end{array}
\end{gather}
where $V_0\in\mathsf{L}^\infty(\Omega_0)$. It is easy to see that the above form is densely defined in $\HS_0$, non-negative and  closed. We denote by $\A_0$ the self-adjoint operator associated with $\a_0$.
It is easy to show that its domain is given by
\begin{gather*}
\dom(\A_0)=
\left\{
f=(f_1,f_2)\in \mathsf{H}^2(\Omega_0\setminus\{0\})\times\C:\;
\begin{array}{l}
f_1(-0)=f_1(+0),\\[1mm] f_1'(+0)-f_1'(-0)=\gamma f_1(\pm 0),\\[1mm]
f_1'(\ell_-)=0\text{ provided }\ell_->-\infty,\\[1mm]
f_1'(\ell_+)=0\text{ provided }\ell_+<\infty
\end{array}
\right\},
\\[1mm] 
(\A_0 f)_1(x)=
-f_1''(x)+V_0(x)f_1(x), \quad (x \ne 0), \qquad
\qquad (\A_0 f)_2=0.
\end{gather*}
Evidently, 
\begin{equation*}
\A_0=\widehat \A_0\oplus 0_{\C}\quad\text{in}\quad\L(\Omega_0)\oplus \C,
\end{equation*}
where $0_\C$ is the null-operator in $\C$, and $\widehat\A_0$ is defined by
the operation 
\begin{equation*}
  -\dfrac{\d^2}{\d x^2}+V_0
  \quad\text{on}\quad (\ell_-,0)\cup(0,\ell_+)
\end{equation*}
with Neumann conditions at $\ell_-$ (provided $\ell_->-\infty$) and 
$\ell_+$ (provided $\ell_+<\infty$) and $\delta$-coupling with strength $\gamma$ at $x=0$.
Consequently,
\begin{equation*}\sigma(\A_0)=\sigma(\widehat\A_0)\cup\{0\}.\end{equation*}

\subsection{Resolvent convergence}

Our first goal is to prove (a kind of) norm resolvent convergence of the operator $\A\e$ to the operator $\A_0$. Since these operators act in different Hilbert spaces $\HS\e=\L(\Omega\e)$ and $\HS_0= \L(\Omega_0)\oplus \C$, respectively, the standard definition of norm resolvent convergence cannot be applied here and should be modified in an appropriate way. The modified definition should be
adjusted in such a way that it still implies  the convergence of spectra as it takes place in the classical situation. The standard approach (see, e.g.\ the abstract scheme in~\cite{IOS89} and its applications to homogenization in perforated spaces~\cite[Chap.~III]{OSY92},~\cite[Chap.~XI]{ZKO94}) is to treat the operator $\mathcal{L}\e \colon \HS_0\to\HS\e$,
\begin{gather*}
\mathcal{L}\e\coloneqq \Res\e \J\e - \J\e \Res_0,
\end{gather*}
where $\Res\e$ and $\Res_0$ are the resolvents of $\A\e$ and $\A_0$, respectively,   and 
$\J\e\colon\HS_0\to \HS\e$ is a suitable bounded linear operator being ``almost isometric'' in a sense 
that
\begin{gather*}
\forall f\in\HS_0\colon\ \lim_{\eps\to 0}\|\J\e f\|_{\HS\e} =\|f\|_{\HS_0}.
\end{gather*}   
 
For the problem we deal in this paper it is natural to define the operator $\J\e$ as follows:
\begin{gather}\label{J}
  (\J\e f)(\x)=
  \begin{cases}
    \eps^{-1/2}f_1(x_1),& \x=(x_1,x_2)\in S\e,\\
    0,&\x\in P\e,\\
    (b\e)^{-1}f_2,&\x\in R\e,
  \end{cases}\quad
  f=(f_1,f_2)\in \HS_0=\L(\Omega_0)\oplus \C.
\end{gather}
The operator $\J\e$ is natural because it is an isometry, namely one has
\begin{gather}\label{CS} 
\forall f\in \HS_0\colon\quad\|\J\e f \|_{\HS\e}=\|f\|_{\HS_0}.
\end{gather} 
Along with $\J\e$ we also introduce the operator $\wt\J\e\colon\HS\e\to\HS_0$ by
\begin{gather}
  \nonumber
  \L(\Omega\e)\ni u\,\longmapsto\wt\J\e u=((\wt\J\e u)_1,(\wt\J\e u)_2)\in \L(\Omega_0)\times\C,\text{ where }\\
\label{J'}
(\wt\J\e u)_1(x_1)=
  \eps^{-1/2}\int_{-\eps}^0 u(x_1,x_2)\d x_2\quad (x_1\in\Omega_0),\qquad
(\wt\J\e u)_2= (b\e)^{-1}\int_{R\e} u(\x)\d\x.
\end{gather}
It is easy to see that $\|\wt\J u\|_{\HS_0}\leq \|u\|_{\HS\e}$.  Moreover, $\wt\J\e$ is the adjoint of $\J\e$, as
\begin{gather}
  \label{J*}
  \forall f\in\HS_0,\
  \forall u\in\HS\e\colon\quad
  (\J\e f,u)_{\HS\e}=(f,\wt\J\e u)_{\HS_0}.
\end{gather}

To guarantee the closeness of the resolvents of the operators $\A\e$ and $\A_0$, the potentials $V\e$ and $V_0$ have to be close in a suitable sense. Namely, we choose the family $\{V\e\}_{\eps>0}$ as follows
\begin{gather}\label{V}
  V\e(\x)=
  \begin{cases}
    V_0(x_1),& \x=(x_1,x_2)\in S\e,\\
    0,&\x\in P\e\cup R\e.
  \end{cases}
\end{gather}
In order to simplify the presentation we assume further that
\begin{equation*}
  V_0(x)\geq 0,
\end{equation*}
and hence both $\A\e$ and $\A_0$ are non-negative operators; the general case needs only slight modifications.  We denote by $\Res\e$ and $\Res_0$ the resolvents of $\A\e$ and $\A_0$, respectively:
\begin{equation*}
  \Res\e\coloneqq (\A\e+\Id)^{-1},\quad \Res_0\coloneqq (\A_0+\Id)^{-1}.
\end{equation*}

We are now in position to formulate the first result  of this work. Below, $\|\cdot\|_{X\to Y}$ stands for the  norm of an operator acting between normed spaces $X$ and $Y$.

\begin{theorem}
  \label{th1}
  One has
  \begin{align}
    \label{th1:1}
    \|\Res\e\J\e- \J\e \Res_0   \|_{\HS_0 \to \HS\e }
    = \|\wt\J\e\Res\e -  \Res_0 \wt\J\e  \|_{\HS\e\to \HS_0}
    &\leq  C_1 \eps^{\min\left\{\al, 1/2-\be\right\}},
  \end{align}
  where $C_1>0$ is a constant independent of $\eps$ (see Remark~\ref{rem:constants} for more details).
\end{theorem}

Obviously, solely the estimate~\eqref{th1:1} provide no information on the closeness of spectra --- simply because~\eqref{th1:1} holds for arbitrary operators $\Res\e\colon\HS\e\to\HS\e$ and $\Res_0\colon \HS_0\to \HS_0$ if we choose $\J\e=0$ and $\wt\J\e=0$.  Therefore, in order to get some information on the closeness of spectra,  we need additional conditions on the operators $\J\e$ and $\wt\J\e$. Such conditions are formulated in the abstract Theorem~\ref{thA:2} below, and also in the original concept of quasi-unitary equivalence in~\cite{P06} and~\cite{P12}, see Section~\ref{sec:quasi-uni}.

\subsection{Spectral convergence}

Our second result concerns Hausdorff convergence of spectra.  Recall (see, e.g.~\cite{Vo02}), 
that for closed sets $X,Y\subset\R$ the \emph{Hausdorff distance} between $X$ and $Y$ is given by
\begin{gather}
  \label{dH}
  d_\Hausdorff (X,Y)
  \coloneqq \max\bigl\{\sup_{x\in X} 
    \inf_{y\in Y}|x-y|;\,\sup_{y\in Y} \inf_{x\in X}|y-x|\bigr\}.
\end{gather}
The notion of convergence provided by this metric   is too restrictive for our purposes.  Indeed, the closeness  of $\sigma(\A\e)$ and $\sigma(\A_0)$  in  the metric $d_\Hausdorff(\cdot,\cdot)$ would mean  that  these  spectra  look nearly the same uniformly on all parts of $[0,\infty)$ --- a situation, which is not guaranteed by norm  resolvent  convergence.
To overcome this difficulty, we introduced the new metric $\wt d_\Hausdorff(\cdot,\cdot)$, which is given by
\begin{gather*}
\wt d_\Hausdorff(X,Y)\coloneqq  d_\Hausdorff
\left(\overline{(1+X)^{-1}},\overline{(1+Y)^{-1}}\right),\ X,Y\subset[0,\infty),
\end{gather*}
where  $(1+X)^{-1}=\{(1+x)^{-1}:\ x\in X\}$ and $(1+Y)^{-1}=\{(1+y)^{-1}:\ y\in Y\}$.
With respect to this metric  two  spectra  can  be  close  even  if  they  differ  significantly  at  high energies. 
Note that $\wt d_\Hausdorff\left(X\e,X\right)\to 0\text{ as }\eps\to 0$ iff
\begin{itemize} 
\item
for each $x \in \R\setminus X$ there exists $d>0$ such that $X\e\cap\{y:\ |y-x|<d\}=\emptyset$ \emph{eventually}  (as $\eps \to 0$), and 

\item  
for any $x\in X$ there exists a family  $\{x\e\}\e$ with $x_\eps \in X_\eps$ such that $\lim_{\eps\to 0}x\e=x$. 

\end{itemize} 
Note that by the spectral mapping theorem
\begin{gather}
\label{dd}
\wt d_\Hausdorff\left(\sigma(\A\e),\sigma(\A_0)\right)= 
d_\Hausdorff\left(\sigma(\Res\e),\sigma(\Res_0)\right).
\end{gather}

\begin{theorem}\label{th2}
  One has
  \begin{gather}\label{th2:est}
    \wt d_\Hausdorff\left(\sigma(\A\e),\sigma(\A_0)\right)
    \leq C_2 \eps^{\min \{\al, 1/2-\be, 2\be\}},
  \end{gather}
  where
  $C_2>0$ is a constant independent of $\eps$, see~\eqref{eq:c2}.
\end{theorem}  
Note that the best convergence rate in \eqref{th2:est} is provided when   $\alpha=1/3$ and $\beta=1/6$.

\begin{remark}\label{rem:deco}
We denote 
\begin{gather}
\label{Dpm}
D^+\e\coloneqq \left\{\x=(x_1,x_2)\in\partial P\e:\ x_2= h\e\right\},\quad
D^-\e\coloneqq \left\{\x=(x_1,x_2)\in\partial P\e:\ x_2=0\right\},
\end{gather}
and $F\e\coloneqq \intr (\overline{S\e\cup P\e})=S\e\cup P\e\cup D^-\e$. 
Let $\A_{F\e}$ be the operator in $\L(F\e)$ acting as $-\Delta + V\e$ and with Neumann boundary conditions on $\partial F\e\setminus D^+\e$ and Dirichlet conditions on $D^+\e$. 
Using similar methods as in the proof of Theorem~\ref{th2}, one can show that 
\begin{gather}
\wt d_\Hausdorff(\sigma(\A_{F\e}), \sigma(\widehat\A_0))\to 0
\quad\text{as}\quad
\eps\to 0;\label{rem:deco:1}
\end{gather}
the $\delta$-coupling at $0$
is caused by the Dirichlet conditions on $D^+\e$ (these boundary conditions can be regarded as an infinite potential). 
Also, let $\A_{B\e}$ be the Neumann Laplacian on $B\e$. 
The first eigenvalue of $\A_{B\e}$ is zero for each $\eps>0$, while the next eigenvalues escape to infinity as $\eps\to 0$. Hence, we get
\begin{gather}\label{rem:deco:2}
\wt d_\Hausdorff(\sigma(\A_{B\e}), \{0\})\to 0
\quad\text{as}\quad
\eps\to 0.
\end{gather}
Finally, in $\L(\Omega\e)=\L(F\e)\oplus \L(B\e)$ we consider the operator $\A\e'\coloneqq\A_{F\e}\oplus\A_{B\e}$.
It follows from \eqref{th2:est}, \eqref{rem:deco:1}-\eqref{rem:deco:2} that the spectra of $\A\e$ and $\A\e'$ are close in the $\wt d_\Hausdorff$-metric as $\eps\to 0$.
The fact that the asymptotic behavior of the spectrum does not change, if one  
detaches the passage from the room and changes the boundary conditions on the contact part of the passage boundary,
is not surprising --- see, e.g.\ the ``Organ pipe Lemma'' in \cite[Sec.~1]{HSS91}.
\end{remark}

\section{Abstract toolbox}
\label{sec3}

In this section we present two abstract results serving to compare the resolvents (Theorem~\ref{thA:1}) and the spectra (Theorem~\ref{thA:2}) of two self-adjoint non-negative operators acting in two different Hilbert spaces.  
The first result was established by the second author in~\cite{P06}, 
and the second result  (in a slightly weaker form) was proven by the first author and G.~Cardone in~\cite{CK19}. For convenience of the reader we will present complete proofs here.
 
Throughout this section $\HS$ and $\wt\HS$ are two  Hilbert spaces, 
$\a$ and $\wt\a$ are closed, densely defined, non-negative
sesquilinear forms in $\HS $ and $\wt\HS$, respectively. We denote by
$\A$ and $\wt\A$ the non-negative, self-adjoint operators associated
with $\a $ and $\wt\a$, by $\Res$ and $\wt\Res$ we denote the resolvents of 
$\A$ and $\wt\A$, respectively:
\begin{equation*}
  \Res\coloneqq (\A+\Id)^{-1},\quad \wt\Res\coloneqq (\wt\A+\Id)^{-1}.
\end{equation*}
Along with $\HS$ and $\wt\HS$  we also introduce spaces  $\HS^1 $ and $\wt\HS^1 $  consisting of functions from  $\dom(\a)=\dom(\A^{1/2})$  and $\dom(\wt\a)=\dom(\wt\A^{1/2})$, respectively,
and equipped with the norms
\begin{gather}\label{scale1}
  \begin{array}{l}
    \|f\|_{\HS^1}
    \coloneqq \|(\A+\Id)^{1/2} f\|_{\HS}
    =\left(\a[f,f]+\|f\|_{\HS}^2\right)^{1/2},\\ 
    \|u\|_{\wt\HS^1}
    \coloneqq \|(\wt\A+\Id)^{1/2}u\|_{\wt\HS}
    =\left(\wt\a[u,u]+\|u\|_{\wt\HS}^2\right)^{1/2},
  \end{array}
\end{gather} 
and spaces  $\HS^2 $ and $\wt\HS^2 $  consisting of functions from  $\dom(\A)$  and $\dom(\wt\A)$, respectively, and equipped with the norms
\begin{gather}\label{scale2}
  \|f\|_{\HS^2}\coloneqq \|(\A+\Id) f\|_{\HS},\qquad 
  \|u\|_{\wt\HS^2}\coloneqq \|(\wt\A+\Id)u\|_{\wt\HS}.
\end{gather} 
Note that
\begin{gather}\label{scale}
  \begin{array}{ccrl}
    \HS^2\subset\HS^1\subset\HS
    &\text{and}&\forall f\in\HS^2\colon
    & \|f\|_{\HS}\leq \|f\|_{\HS^1}\leq \|f\|_{\HS^2},\\[2mm]
    \wt\HS^2\subset\wt\HS^1\subset\wt\HS
    &\text{and}&\forall u\in\wt\HS^2\colon
    & \|u\|_{\wt\HS}\leq \|  u\|_{\wt\HS^1}\leq \|  u\|_{\wt\HS^2}.
  \end{array}
\end{gather}
Moreover, due to the non-negativity of $\A$ and $\wt\A$ one has the estimates
\begin{gather}
  \label{scale+}
  \forall f\in\dom(\A)\colon\
  \|\A f\|_{\HS}\leq \|  f\|_{\HS^2},\qquad
  \forall u\in\dom(\wt\A)\colon\
  \|\wt\A u\|_{\wt\HS}\leq \|  u\|_{\wt\HS^2}.
\end{gather}
 
\subsection{Resolvent convergence} 
 
\begin{theorem}
  \label{thA:1} 
  Let
  \begin{equation*}
    \J \colon \HS\to \wt{\HS},\quad 
    \wt\J \colon \wt{\HS}\to \HS,\quad 
    \J ^1 \colon {\HS^1 } \to {\wt\HS^1},\quad
    \wt\J ^{1} \colon {\wt\HS^1}\to {\HS^1}
  \end{equation*}
  be linear operators satisfying the conditions 
  \begin{align}
    \label{thA:1:1}
    \|\J f-\J^1 f\|_{\wt\HS}
    &\leq \delta\|f\|_{\HS^1 }&& \forall f\in \HS^1,\\
    \label{thA:1:2}   
    \|\wt\J u-\wt\J^{1} u\|_{ \HS}
    &\leq \delta\|u\|_{\wt\HS^1}&& \forall u\in \wt\HS^1,\\
    \label{thA:1:3}
    \left|(\J f,u)_{\wt\HS}-(f, \wt\J u)_{\HS}\right|
    &\leq \delta\|f\|_{ \HS}\|u\|_{\wt\HS}&&\forall f\in\HS, u\in\wt\HS,\\
    \label{thA:1:4}
    \left|\wt\a[\J^1 f,u]-\a[f, \wt\J^{1} u] \right|
    &\leq \delta\|f\|_{\HS^2 }\|u\|_{\wt\HS^1}&& \forall f\in\HS^2, u\in\wt\HS^1  
  \end{align}
  for some  $\delta\geq 0$. Then 
  \begin{align}
    \label{thA:diff:1}
    \|\wt\Res\J -\J \Res\|_{\HS\to \wt\HS}&\leq 4\delta.
  \end{align}
\end{theorem} 

\begin{proof}
  Let $g\in \HS$ and $v\in\wt\HS$ be arbitrary. We set $f\coloneqq \Res g\in\dom(\A)$, $u\coloneqq \wt\Res v\in\dom(\wt\A)$. One has
  \begin{align*} 
    ((\wt\Res\J  -\J\Res) g,v )_{\wt\HS}
    &= (\J g,\wt\Res v )_{\wt\HS}-(\J\Res g,v )_{\wt\HS}\\
    &= (\J \A f, u)_{\wt\HS}-(\J f, \wt\A u )_{\wt\HS}\\
    &= 
      \underbrace{(\J\A f,u)_{\wt\HS} - (\A f, \wt \J u)_{\HS}}_{I_1\coloneqq }
      +\underbrace{(\A f,(\wt\J  - \wt \J^1) u)_{\HS}}_{I_2\coloneqq }\\
    &\hspace{0.15\textwidth}
      +\underbrace{\a[f,\wt\J^1u]-\wt \a[\J^1 f,  u]}_{I_3\coloneqq } 
    +\underbrace{((\J^1-\J)f,\wt \A u))_{\wt\HS}}_{I_4\coloneqq }
  \end{align*}
  Using~\eqref{thA:1:1}--\eqref{thA:1:4} one can estimate all terms in the right-hand-side of the above equality,
  \begin{gather*}
    \begin{array}{ll}
      |I_1|\leq \delta\|\A f\|_{\HS}\|u\|_{\wt\HS},
      &|I_2|\leq \delta\|\A f\|_{\HS}\|u\|_{\wt\HS^1},\\
      |I_3|\leq \delta\|f\|_{\HS^2}\|u\|_{\wt\HS^1},
      &|I_4|\leq \delta\|f\|_{\HS^1}\|\wt\A u\|_{\wt\HS},
    \end{array}
  \end{gather*}
  hence, taking  into account~\eqref{scale}--\eqref{scale+}, we arrive at the estimate
  \begin{gather}
    \label{RRfg:abs}
    \forall g\in\HS,\, v\in\wt\HS \colon\quad
    |((\wt\Res\J  -\J\Res) g,v )_{\wt\HS}|
    \leq 4\delta \|f\|_{\HS^2}\|u\|_{\wt\HS^2}
    = 4\delta \|g\|_{\HS}\|v\|_{\wt\HS}.
  \end{gather}
  Evidently,~\eqref{RRfg:abs} implies~\eqref{thA:diff:1}.  The theorem is proven.
\end{proof}

\begin{remark}
It is well-known~\cite[Theorem~VI.3.6]{Ka66}, that convergence of sesquilinear forms with \emph{common domain} implies norm resolvent convergence of the associated operators (see the recent paper~\cite[Theorem~2]{BF18} for a quantitative version of this result); in these theorems
the convergence of the forms $\a\e$ to the form $\a$ means that the inequality  
\begin{gather}\label{form:dist:1}
\left|\a\e[f,f]-\a[f,f] \right|\leq 
  \delta\e \left(\a[f,f]+\|f\|^2_{\HS}\right),\quad \delta\e\to 0,
\end{gather}
holds for each $f\in\dom(\a\e)=\dom(\a)$.
In this sense, Theorem~\ref{thA:1} can be regarded as a generalization of~\cite[Theorem~VI.3.6]{Ka66},~\cite[Theorem~2]{BF18}   to the setting of varying spaces.
\end{remark}

\begin{remark}
It is easy to see from the proof above, that
some of the conditions~\eqref{thA:1:1}--\eqref{thA:1:4} can be weakened.
For example, Theorem~\ref{thA:1}
remains valid if \eqref{thA:1:4} is substituted 
by 
\begin{gather}\label{form:dist:2}
  \left|\wt\a[\J^1 f,u]-\a[f, \wt\J^{1} u] \right|\leq 
  \delta\|f\|_{\HS^2 }\|u\|_{\wt\HS^2},\quad 
  \forall f\in \HS^2 ,\ u\in \wt\HS^2.
\end{gather} 
Nevertheless, in most of the  applications one is able to establish stronger estimate~\eqref{thA:1:4}
(cf. Lemma~\ref{lm2}). Moreover, sometimes (for example, when studying convergence of graph-like manifolds~\cite{P06,P12}), one even can prove the stronger 
inequality
\begin{gather*}
  \Bigl|\wt\a[\J^1 f,u]-\a[f, \wt\J^{1} u] \Bigr|
  \leq \delta\|f\|_{\HS^1 }\|u\|_{\wt\HS^1},\quad 
  \forall f\in \HS^1 ,\ u\in \wt\HS^1,
\end{gather*} 
which can be regarded as a counterpart to~\eqref{form:dist:1}.
\end{remark} 

\begin{remark}
Usually in  applications the operators $\J$ and $\wt\J$ appear in a natural way (as, for example, $\J$ defined in~\eqref{J} and $\wt\J$ defined in~\eqref{J'} in our case), while the other two operators $\J^1$ and $\wt\J^1$ should be constructed as ``almost'' restrictions of $\J$ and $\wt\J$ to $\HS^1$ and $\wt\HS^1$, respectively, modified in such a way that they respect the form domains (see conditions~\eqref{thA:1:1}--\eqref{thA:1:2} above).
\end{remark} 

\subsection{Spectral convergence} 

Recall that  the Hausdorff distance  $d_\Hausdorff(\cdot,\cdot)$ is defined via~\eqref{dH}. 
It is well-known, that norm convergence of bounded self-adjoint operators in a \emph{fixed} Hilbert space implies Hausdorff convergence of spectra of the underlying resolvents.
Namely, let $\Res$ and $\wt\Res$ be bounded self-adjoint operators in a Hilbert space $\HS$, then~\cite[Lemma~A.1]{HN99}
\begin{equation}
  \label{eq:spec.est.hn99}
  d_\Hausdorff(\sigma(\Res),\sigma(\wt\Res))
  \leq \|\Res-\wt\Res\|_{\HS\to\HS}
\end{equation}
(in fact, the above results holds even for normal operators).
Our goal is to find an analogue of this result for the case of operators acting in different Hilbert spaces.
In what follows, we assume that the operators $\wt\A$ and $\A$ are unbounded, whence, $0\in \sigma(\wt\Res)\cap\sigma(\Res)$.

\begin{theorem}\label{thA:2}
  Let  
  $\J \colon\HS\to \wt{\HS}$, $\wt\J \colon\wt\HS\to {\HS} $ be  linear bounded operators satisfying 
  \begin{align}
    \label{thA:2:1}
    \|\wt\Res\J - \J \Res \|_{\HS\to \wt\HS}&\leq \eta,\\
    \label{thA:2:3}
    \|\wt\J\wt\Res - \Res\wt \J  \|_{\wt\HS\to  \HS}&\leq \wt\eta,
  \end{align}
  and, moreover,
  \begin{align}
    \label{thA:2:2}
    \|f\|^2_{ \HS}&\leq \mu \|\J  f\|^2_{\wt\HS}+\nu \,  \a[f,f],
                    \quad \forall f\in \dom( \a),\\
    \label{thA:2:4}
    \|u\|^2_{\wt\HS}&\leq \wt\mu \|\wt\J  u\|^2_{\HS}+\wt\nu \, \wt\a[u,u],
                      \quad \forall u\in \dom( \wt\a),
  \end{align}
  for some positive constants $\eta$, $\mu$, $\nu$, $\wt\eta$,
  $\wt\mu$ and $\wt\nu$.  Then for any $\kappa$, $\wt\kappa \in (0,1)$
  we have
  \begin{gather*}
    d_\Hausdorff\left(\sigma(\Res),\, \sigma(\wt\Res)\right)
    \leq 
    \max\left\{
      \eta \sqrt{\frac\mu\kappa};\,
      \frac{\nu}{1-\kappa};\,
      \wt\eta \sqrt{\frac{\wt\mu}{\wt\kappa}};\,
      \frac{\wt\nu}{1-\wt\kappa}
    \right\}.
  \end{gather*}
\end{theorem} 

\begin{remark}
  \indent
  \begin{enumerate}
  \item In a typical application, $\mu \ge 1$ is close to $1$, and
    $\nu>0$ is small, as if $\J$is an isometry, then $\mu=1$ and
    $\nu=0$.  A similar remark holds for $\wt \mu$ and $\wt \nu$.  In
    our application later, we have $\mu=1$ and $\nu=0$ (hence, we are allowed to take $\kappa=1$ by a limit argument), and
    $\wt \mu=\mu_\eps=1+ C_6 \eps^{2\alpha}$ and
    $\wt \nu= \nu_\eps=C_{17} \eps^{2\min\{\alpha,\beta\}}$ (see
    Lemma~\ref{lm3}).

  \item Note that the upper bound $\eta \sqrt{\mu/\kappa}$ arises from
    values of $\sigma(\Res)$ closely below $1$, i.e.\ from \emph{small}
    values of $\sigma(\A)$ whereas the upper bound $\nu/(1-\kappa)$
    arises from values of $\sigma(\Res)$ near $0$, i.e.\ from
    \emph{large} values of $\sigma(\A)$. A similar remark holds for
    $\wt \Res$ and $\wt\A$.

  \item The role of $\kappa$ (and $\wt \kappa$) is as follows: One
    can, of course, fix $\kappa=\wt \kappa=1/2$, then the error is the
    maximum of $\eta \sqrt{2\mu}$, $2\nu$, $\wt \eta \sqrt{2\wt \mu}$
    and $2\wt \nu$.  Herbst and Nakamura proved in the classical case
    an estimate of the Hausdorff distance of the spectra in terms of
    $\norm{\Res-\wt \Res}_{\HS \to \HS}$,
    cf.~\eqref{eq:spec.est.hn99}.  If we are aiming in a similar
    result, we have to use the two norms
    in~\eqref{thA:2:1}--\eqref{thA:2:3}.  The constants $\kappa$ and
    $\wt \kappa$ allow to estimate the Hausdorff distance of the
    spectra in terms of $\eta$ and $\wt \eta$ \emph{with a constant as
      close to $1$ as wanted}.  The price of this factor to be close
    to $1$ is then a worse estimate in the second term, namely
    $\nu/(1-\kappa)$.
  \end{enumerate}
\end{remark}

\begin{proof}[Proof of Theorem~\ref{thA:2}] 
  For each $z\in\C$ one has the   estimate
  \begin{gather}\label{Haus1}
    \forall\phi\in\wt\HS\setminus\{0\} \colon \quad 
    \dist(z,\sigma(\wt\Res ))
    \leq \frac{\|(\wt\Res -z\Id)\phi\|_{\wt\HS}}{\|\phi\|_{\wt\HS}}
  \end{gather}
  (hereinafter for $x\in\R$ and a compact set $Y\subset\R$ we denote $\dist(x,Y)\coloneqq \inf_{y\in Y}|x-y|$).
  Indeed, for $z\in \sigma(\wt\Res)$  estimate~\eqref{Haus1} is trivial, while for 
  $z \in \C\setminus\sigma(\wt\Res)$ it follows easily from
  \begin{equation*}
    \|(\wt\Res -z\Id)^{-1}\|_{\wt\HS}
    =\frac 1 {\dist(z,\,\sigma(\wt\Res))}.
  \end{equation*}
  
  In what follows we assume that $z\in\sigma(\Res )\cap[ L_\nu,1]$, where 
  \begin{gather}
    \label{Lnu}
    L_\nu\coloneqq  \frac{\nu}{\nu+(1-\kappa)}\in (0,1).
  \end{gather}
  We denote $\lambda_z\coloneqq \frac{1-z}z$. It is easy to see that the following identity holds:
  \begin{gather}
    \label{identity}
    (\Res -z\Id)\psi =-z\, \Res(\A-\lambda_z\Id )\psi,\quad\psi\in\dom(\A).
  \end{gather}
  Moreover, by the spectral mapping theorem $\lambda_z\in \sigma(\A) $ and hence 
  \begin{gather}\label{weyl}
    \forall{\rho}>0\quad \exists\psi_{\rho}\in\dom(\A)\colon\quad
    \|\psi_{\rho}\|_{\HS}  =1,\quad 
    \|(\A-\lambda_z\Id)\psi_{\rho}\|_{\HS}\leq {\rho}. 
  \end{gather} 
  Taking into account that $z\in (0,1]$ and $\|\Res\|_{\HS\to\HS}\leq 1$,
  one gets   from~\eqref{identity}--\eqref{weyl}
  \begin{gather}\label{Haus2}
    \|(\Res -z\Id)\psi_{\rho}\|_{\HS}\leq   {\rho}.
  \end{gather}
  
  Using~\eqref{thA:2:2},~\eqref{weyl} and taking into account that 
  \begin{gather}\label{lambdaz}
    \lambda_z\leq \frac{1-L_\nu}{L_\nu},
  \end{gather} 
  one can prove that $\J\e\psi_{\rho}\not=0$ for small enough ${\rho}$.
  Indeed,    we get
  \begin{align}\nonumber
    \mu\| \J\psi_{\rho}\|_{\wt\HS}^2
    &\overset{\eqref{thA:2:2}}\geq \|\psi_{\rho}\|_{\HS}^2-\nu\,\a[\psi_{\rho},\psi_{\rho}]
    =
    (1-\lambda_z \nu)\|\psi_{\rho}\|_{\HS}^2 -\nu(\A\psi_{\rho}
      -\lambda_z\psi_{\rho},\psi_{\rho})_{\HS}\\
    \label{Jnot0}
    &\overset{\eqref{weyl},\,\eqref{lambdaz}}\geq
    1-\frac{1-L_\nu}{L_\nu}\nu-{\rho}\nu
    = \kappa-\rho\nu
  \end{align}
  (we use \eqref{Lnu} for the last equality).
   Hence $\J\e\psi_{\rho}\not=0$
  as $\rho<\kappa/\nu$.
  
  For $z\in\sigma(\Res)\cap[L_\nu,1]$, ${\rho}\in (0,\kappa/\nu)$ we obtain
  using~\eqref{Haus1},~\eqref{Haus2} and~\eqref{Jnot0}: 
  \begin{align*}
    \dist(z,\,\sigma(\wt\Res)) 
    &\leq  
      \frac{\|(\wt\Res-z\Id) \J \psi_{\rho}\|_{\wt\HS}}
      {\| \J \psi_{\rho}\|_{\wt\HS}}\\
    &\leq 
      \frac{\|(\wt\Res\J - \J   \Res)\psi_{\rho}\|_{\wt\HS}+
      \| \J (\mathcal{R} - z\Id)\psi_{\rho}\|_{\wt\HS}}
      {\| \J \psi_{\rho}\|_{\wt\HS}} 
      \leq
      \frac{\eta+\|\J\|_{\HS\to\wt\HS} \cdot{\rho}}
      {\sqrt{\mu^{-1}\left(\kappa-{\rho}\nu\right)}}.
  \end{align*}
  Passing to the limit $\rho\to 0$   we arrive at the estimate
  \begin{gather}\label{Haus3}
    \dist(z,\sigma(\wt\Res))
    \leq  
    {\eta \sqrt{\frac \mu\kappa}},\quad
    \forall z\in \sigma(\mathcal{R})\cap[L_\nu,1].
  \end{gather} 
  Finally, taking into account that $0\in\sigma(\wt\Res)$ we also get
  \begin{gather}\label{Haus4}
    \dist(z,\sigma(\wt\Res))\leq \dist(z,0)\leq
    L_\nu,\quad \forall z\in \sigma(\Res)\cap[0,L_\nu].
  \end{gather}
  Combining~\eqref{Haus3}--\eqref{Haus4} and $L_\nu \le \nu/(1-\kappa)$ we obtain
  \begin{gather}\label{Haus5}
    \dist(z,\sigma(\wt\Res))
    \leq
    \max\left\{\eta \sqrt{\frac\mu\kappa};\,
      \frac \nu{1-\kappa}\right\},\quad \forall z\in \sigma(\Res) .
  \end{gather}
  Repeating verbatim the above arguments we also obtain the estimate
  \begin{gather}\label{Haus6}
    \dist(z,\sigma( \Res))
    \leq
    \max\left\{\wt\eta \sqrt{\frac{\wt\mu}{\wt \kappa}};\,
      \frac{\wt\nu}{1-\wt\kappa}\right\},\quad \forall z\in \sigma(\wt\Res).
  \end{gather}
  The statement of the theorem follows immediately from~\eqref{Haus5}--\eqref{Haus6} and~\eqref{dH}. 
\end{proof}

\subsection{Quasi-unitary operators} 
\label{sec:quasi-uni}

Let us here finally comment on the concept originally introduced in~\cite{P06} and~\cite{P12}:
\begin{definition}
  \label{def:quasi-uni}
  We say that $\J$ and $\wt \J$ are \emph{$\delta$-quasi-unitary} for some $\delta \ge 0$ if
  \begin{align}
    \label{eq:quasi-uni:1}
    \norm{f - \wt\J \J f}_\HS &\le \delta \norm f_{\HS^1}, 
    && \forall f\in \HS^1,\\
    \label{eq:quasi-uni:2}
    \norm{u - \J \wt\J u}_\HS &\le \delta \norm u_{\wt \HS^1},
    && \forall u\in \wt\HS^1,
  \end{align}
  and also~\eqref{thA:1:3} holds.
  We say that $\A$ and $\wt\A$ are \emph{$\delta$-quasi-unitarily equivalent}, if (additionally to~\eqref{eq:quasi-uni:1}--\eqref{eq:quasi-uni:2})  $\J$ and $\wt \J$ also fulfil
 \begin{gather*}
    \|\wt\Res\J -\J \Res\|_{\HS\to \wt\HS}\leq \delta,\qquad
    \|\wt\J\wt\Res -\Res\wt\J \|_{\wt\HS\to  \HS}\leq \delta.
  \end{gather*}  
\end{definition}
The above concept allows to generalise norm resolvent convergence in
the sense that $\A_\eps$ converges to $\A_0$ in \emph{generalised norm
  resolvent sense} (with convergence speed $\delta_\eps$) if $\A_\eps$
and $\A_0$ are $\delta_\eps$-quasi-unitarily equivalent with
$\delta_\eps \to 0$ as $\eps \to 0$ (cf.~\eqref{HHAA}).  This concept generalises the classical norm resolvent convergence in the sense that if
$\HS_\eps=\HS_0=\HS$ and choosing $\J=\wt\J=\Id$ the identity operator
on $\HS$, then the generalised norm resolvent convergence is just the
classical norm resolvent convergence
$\norm{\Res_\eps - \Res_0}_{\HS \to \HS} \le \delta_\eps \to 0$ as
$\eps \to 0$.

 As for the classical norm resolvent convergence we have~\cite{P06,P12} (see also~\cite{PS19} for a brief up-to-date version and more details):
 \begin{proposition}[{\cite[Sec..~1.3]{PS19},~\cite[App.~A.4--A.5]{P06}}]
   \label{prp:conv.op.fcts}
   If $\A_\eps$ converges to $\A_0$ in generalised norm resolvent convergence with convergence speed $\delta_\eps$, then we have
   \begin{equation*}
     \norm{\Psi(\A_\eps )-\J_\eps \Psi(\A_0) \wt \J_\eps}_{\HS\e\to\HS\e} \to 0
   \end{equation*}
   as $\eps \to 0$ for suitable functions $\Psi$ (e.g.\ measurable and continuous in a neighbourhood $U$ of $\sigma(\A_0)$ with $\Psi(\lambda)(\lambda+1)^{1/2}\to 0$ as $\lambda \to \infty$.  If $\Psi$ is holomorphic on $U$, then the above norm of the resolvent difference is of order $\delta_\eps$. 
 \end{proposition}
The above proposition applies in particular to the heat operator with $\Psi=\Psi_t$ and $\Psi_t(\lambda)=\euler^{-\lambda t}$ or spectral projections $\Psi=\1_I$ and $\partial I \cap \sigma(\A_0)=\emptyset$.  We also showed spectral convergence, see~\cite{P06,P12} for details.  Nevertheless, the result in Theorem~\ref{thA:2} is more explicit as it imitates the proof of~\eqref{eq:spec.est.hn99} of Herbst and Nakamura~\cite{HN99} and gives better error estimates.

The above concept of quasi-unitary operators implies the spectral convergence as in Theorem~\ref{thA:2}:
\begin{proposition}
  \label{prp:quasi-uni-spec-conv}
  Assume that $\J$ and $\wt\J$ are $\delta$-quasi-unitary with $\delta < 2/3$.  Then the assumptions~\eqref{thA:2:2}--\eqref{thA:2:4} in Theorem~\ref{thA:2} are fulfilled with
  \begin{equation*}
    \mu=\wt \mu
    =1+\frac{4\delta}{2-3\delta}
      \qquad\text{and}\qquad
    \nu=\wt \nu
    =\frac \delta{2-3\delta}.
  \end{equation*}
\end{proposition}
\begin{proof}
  We have
  \begin{align*}
    \norm f_\HS^2 - \norm{\J f}_{\wt \HS}^2
    & =(f-\wt \J\J f,f)_{\HS}
      +(\wt \J \J f,f)_{\HS}-(\J f,\J f)_{\wt\HS}
    & \text{(using~\eqref{eq:quasi-uni:1} and~\eqref{thA:1:3})}\\
    &\le \delta \norm f_{\HS^1} \norm f_\HS
      +  \delta \norm f_\HS \|\J f\|_{\wt\HS} \\
    &\le \frac \delta 2 \bigl(\norm f_{\HS^1}^2 + \norm f_\HS^2 \bigr)
      +  \frac \delta 2 \left(\norm f_\HS^2+ \|\J f\|_{\wt\HS}^2\right)\\
    & = \frac {3\delta} 2  \norm f_\HS^2 + \frac \delta 2 \a[f,f]
      + \frac \delta 2  \|\J f\|_{\wt\HS}^2,
  \end{align*}
  whence we get the desired estimate
  \begin{equation*}
    \norm f_\HS^2 
    \le \left(1+\frac {4\delta}{2-3\delta}\right) \norm{\J f}_{\wt \HS}^2
      + \frac \delta {2-3\delta} \a[f,f]
  \end{equation*}
  provided $\delta < 2/3$.  The estimate~\eqref{thA:2:4} follows similarly.
\end{proof}

\begin{remark}
\label{rem:why-not-quasi-uni}
In our concrete example (cf.~\eqref{HHAA}), $\J=\J_\eps$ is an isometry (cf.~\eqref{CS}), hence~\eqref{thA:2:2} follows with $\mu=1$ and $\nu=0$.

Note that showing~\eqref{thA:2:4} \emph{directly} (as in Lemma~\ref{lm3}), we obtain $\wt \mu=\mu_\eps=1+C_6 \eps^{2\alpha}$ and $\wt \nu=\nu_\eps = C_{17} \eps^{2\min\{\alpha,\beta\}}$, whereas applying Proposition~\ref{prp:quasi-uni-spec-conv} together with Lemma~\ref{lm4}, we only have $\wt \mu=\mu_\eps=1+\mathcal O(\eps^{\min\{\alpha,\beta\}})$ and $\wt \nu=\nu_\eps = \mathcal O(\eps^{\min\{\alpha,\beta\}})$, hence a worse estimate for the spectral convergence.  
\end{remark}

\section{Proof of the main results\label{sec4}}

\subsection{Preliminaries}
For the proof of Theorems~\ref{th1} and~\ref{th2} we will   use abstract results given in Section~\ref{sec3} (Theorems~\ref{thA:1} and~\ref{thA:2}, respectively). Recall, that these abstract results serve to compare the resolvents and spectra of self-adjoint non-negative operators $\wt\A$ and $\A$ acting in different Hilbert spaces $\wt\HS$ and $\HS$, respectively.
We will apply these abstract theorems for  
\begin{gather}
  \label{HHAA}
  \wt\HS=\HS\e,\quad \HS=\HS_0, \quad \wt\A=\A\e,\quad \A=\A_0 
\end{gather}
(recall that $\HS\e=\L(\Omega\e)$ and $\HS_0=\L(\Omega_0)\oplus\C$, and $\A\e$ and $\A_0$ are the self-adjoint operators acting in these spaces associated with the sesquilinear forms $\a\e$~\eqref{ae} and $\a_0$~\eqref{ae0}).

Similarly to~\eqref{scale1}--\eqref{scale2} we introduce Hilbert spaces $\HS\e^k$ and $\HS_0^k$, $k=1,2$, consisting of functions $u\in\dom(\A\e^{k/2})$ and $f\in\dom(\A_0^{k/2})$, respectively, equipped with the norms
\begin{equation*}
  \|u\|_{\HS\e^k}=\|(\A\e+\Id)^{k/2}u\|_{\HS\e},\quad
  \|f\|_{\HS_0^k}=\|(\A_0+\Id)^{k/2}f\|_{\HS_0}.
\end{equation*} 
Note that for Sobolev spaces we use a sans serif font, e.g.\ $\H^1(\Omega\e)$, $\H^2(\Omega_0)$, etc.

Recall, that the sets $D\e^\pm$ are given in~\eqref{Dpm}.
We introduce several other subsets of $\Omega\e$:
\begin{align*}
 D\e^0\coloneqq \left\{\x=(x_1,x_2)\in S\e:\ 
         x_1=0\right\},\quad
  Y\e\coloneqq \left\{\x=(x_1,x_2)\in S\e:\ 
       |x_1|<\frac\eps2 \right\}.
\end{align*}
Note that $Y\e\subset S\e$ due to~\eqref{eq:def.eps0}.
We also denote
\begin{gather*}
  \Omega_0^\pm\coloneqq \{x\in\Omega_0:\ \pm x>0\},\quad
  \wt\Omega_0\coloneqq \Omega_0\cap \left[-\frac12,\frac12\right],\quad
  \wt\Omega_0^\pm\coloneqq \Omega_0^\pm \cap \wt\Omega_0.
\end{gather*}

By $\la u\ra _{D}$ we denote the mean value of the function $u(\x)$ in the domain $D$, i.e.\
\begin{equation*}
  \la u\ra_{D}=|D|^{-1}\int_{D}u(\x)\d \x,
\end{equation*}
where $\abs D$ denotes the area of $D$.  Also we keep the same notation if $D$ is a segment (for example, $D\e^0$); in this case we integrate with respect to the natural coordinate on this segment, and $|D|$ denotes its length.
\smallskip

In the following,  we need the standard   Sobolev inequality.

\begin{lemma}
\label{lemma:sobolev}
Let $\mathcal{I}$ be a bounded interval.
One has
  \begin{gather}
  \label{Sobolev1}
    \forall f\in \H^1(\mathcal{I})\colon\quad
    \|f\|^2_{\mathsf{L}^\infty(\mathcal{I})}\leq  
    \ell_{\mathcal{I}}\|f\|^2_{\H^1(\mathcal{I})},
  \end{gather}  
  where  the constant $\ell_{\mathcal{I}}>1$ depends only on the length $|\mathcal{I}|$ of $\mathcal{I}$.
\end{lemma}

\begin{remark}
One can prove, using arguments as in \cite[Sec.~6.1]{P16}, that \eqref{Sobolev1} holds with $\ell_{\mathcal{I}}=\coth(\ell/2)$. 
\end{remark}
 
The following auxiliary estimates will be used further in the proof of Theorems~\ref{th1}--\ref{th2}.
Similar estimates can be found in~\cite[Lemmata~3.1,\,3.3]{CK15} and~\cite[Lemmata~3.1,\,5.2]{CK17}.

\begin{lemma}\label{lm:aux}
  For any $u\in\H^1(\Omega\e)$ one has
  \begin{gather}
    \label{lm:aux:2}
    \abs{\la u \ra_{D^+\e}-\la u \ra_{R\e}}
    \leq C_3\abs{\ln \eps}^{1/2}\|\nabla u\|_{\L(R\e)},\\
    \label{lm:aux:2'}
    |\la u \ra_{D^-\e}-\la u \ra_{Y\e}|
    \leq C_4 \abs{\ln \eps}^{1/2}\|\nabla u\|_{\L(Y\e)},\\
    \label{lm:aux:1}
    \abs{\la u \ra_{D^-\e}-\la u \ra_{D\e^0}}
    \leq C_5 \abs{\ln \eps}^{1/2}\|\nabla u\|_{\L(Y\e)},\\
    \label{lm:aux:3}
    \|u\|^2_{\L(P\e)}
    \leq C_6\eps^{2\al}\Bigl(\|u\|^2_{\L(S\e)} + \norm{\nabla u}_{\L(S\e)}^2 
      + \|\nabla u\|_{\L(P\e)}^2\Bigr).
  \end{gather}
\end{lemma}
 
\begin{proof}
The following estimates were established in~\cite[Ineqs.~(3.14),\,(3.13),\,(5.16)]{CK17}: 
\begin{gather}
    \label{aux:2}
    |\la u \ra_{D^+\e}-\la u \ra_{R\e}|
    \leq C_7 \abs{\ln d\e}^{1/2}\|\nabla u\|_{\L(R\e)},\\
    \label{aux:1}
    |\la u \ra_{D^-\e}-\la u \ra_{Y\e}|
    \leq C_8 \abs{\ln d\e}^{1/2}\|\nabla u\|_{\L(Y\e)},\\
    \label{aux:3}
    \|u\|^2_{\L(P\e)}\leq 
    C_9 h\e\left(d\e\eps^{-2}\|u\|^2_{\L(Y\e)}
      +d\e \abs{\ln d\e}\norm{\nabla u}_{\L(Y\e)}^2\right)+ 2(h\e)^2 \|\nabla u\|_{\L(P\e)}^2
\end{gather}
for any $u\in\H^1(\Omega\e)$.
Note that~\cite{CK17} deals with ``rooms'' and ``passages'' of a different size than in the current work, however, the estimates~\eqref{aux:2}--\eqref{aux:3} were established for arbitrary $\eps$,~$d\e$, $h\e$, $b\e$
(such that $d\e<\eps\leq\eps_0<1$, $d\e<b\e\le 1$), and the constants $C_7$, $C_8$ and $C_9$ depend only on $\eps_0$.
 
Due to $\eps\leq \min\{\ga,\ga^{-1}\}$ (cf.~\eqref{eq:def.eps0}), one has  
$|\ln\gamma|\leq|\ln\eps|$. Using this and taking into account that $d\e=\gamma\eps^{\al+1}$ with $\al>0$,
we deduce from~\eqref{aux:2} the required estimate~\eqref{lm:aux:2} with $C_3=(\alpha+2)^{1/2}C_7$.  Similarly,~\eqref{lm:aux:2'} holds with $C_4=(\alpha+2)^{1/2}C_8$.

Before proving the estimate~\eqref{lm:aux:1} we need an additional step.  Let $u\in C^\infty(\overline{Y\e})$. One has:
\begin{gather*}
u(x_1,x_2)=u(0,x_2)+\int_0^{x_1}\partial_1 u(\tau,x_2) \d\tau,
\end{gather*}
where $x_1\in \left(-\frac \eps2, \frac\eps2\right)$,
$x_2\in (-\eps,0)$, and $\partial_1 u$ denotes the partial derivative of $u$ with respect to the first variable.  Integrating the above equality with respect to $x_1$ over $(-\frac\eps2,\frac\eps2)$, with respect to $x_2$ over $(-\eps,0)$, and then dividing by $\eps^2$, we get
\begin{gather*}
  \la u \ra_{Y\e}=
  \la u \ra_{D\e^0}+
  \eps^{-2}\int_{-\eps}^0\int_{-\eps/2}^{\eps/2}\int_0^{x_1}
  \partial_1 u(\tau, x_2) \d\tau\d x_1\d x_2,
\end{gather*}
whence, using the Cauchy-Schwarz inequality, we deduce the estimate
\begin{gather}\label{aux:4} 
  |\la u \ra_{Y\e}-
  \la u \ra_{D\e^0}|\leq \frac12
  \|\nabla u\|_{\L(Y\e)}.
\end{gather}
By standard density arguments~\eqref{aux:4} holds not only for smooth functions, but also for any $u\in\H^1(Y\e)$.
The  estimate~\eqref{lm:aux:1} follows from~\eqref{lm:aux:2'} and~\eqref{aux:4} with $ C_5=C_4+\frac12\abs{\ln\eps_0}^{-1/2}$.

It remains to prove the estimates~\eqref{lm:aux:3}. 
Let $u\in C^\infty(\overline{S\e})$. 
By Lemma~\ref{lemma:sobolev} one has 
\begin{gather*}
|u(x_1,x_2)|^2
\leq \ell_{\wt\Omega_0}
\|u(\cdot,x_2)\|^2_{\L(\wt\Omega_0)},
\end{gather*}
where $x_1\in \wt\Omega_0$, $x_2\in (-\eps,0)$.
Integrating the above inequality with respect to $x_1$ over $(-\frac\eps2,\frac\eps2)$, and with respect to  $x_2$ over $(-\eps,0)$,  we obtain
\begin{gather}
\label{aux:5}
    \|u\|^2_{\L(Y\e)}
    \leq \ell_{\wt\Omega_0}\eps \| u\|^2_{\H^1(\wt\Omega_0\times (-\eps,0))}
    \leq \ell_{\wt\Omega_0}\eps \| u\|^2_{\H^1(S\e)};
\end{gather}
by  density arguments the estimate \eqref{aux:5} holds for any $u\in\H^1(S\e)$.
Combining~\eqref{aux:3},~\eqref{aux:5} and using~\eqref{dhb}, we arrive at the desired estimate~\eqref{lm:aux:3} with  
\begin{equation}
  \label{C6}
  C_6 \coloneqq  
  \max\left\{C_9\gamma(\ell_{\wt\Omega_0}+ \al+2);\,2\right\}.
\qedhere
\end{equation}
\end{proof}

\subsection{Proof of Theorem~\ref{th1}}

In order to utilise Theorem~\ref{thA:1} we need to construct suitable
operators
\begin{equation*}
  \J\e^1\colon\HS_0^1\to \HS\e^1,\qquad 
  \wt\J\e^1\colon\HS\e^1\to \HS_0^1,
\end{equation*}
 where $\HS^1_0$ and $\HS^1\e$ are the  energy spaces  associated with the forms $\a_0$ and $\a\e$, i.e.\
\begin{gather*}
  \begin{array}{lll}
    \HS^1_0=\dom(\a_0)&\text{equipped with the norm}
    &\|f\|_{\HS^1_0}=(\a_0[f,f]+\|f\|_{\HS_0}^2)^{1/2},
    \\[1mm]
    \HS\e^1=\dom(\a\e)&\text{equipped with the norm}
    &\|u\|_{ \HS\e^1}=( \a\e[u,u]+\|u\|_{\HS\e}^2)^{1/2}.
  \end{array}
\end{gather*} 

We define the operator $\J\e^1$ as follows. Let
$f=(f_1,f_2)\in \dom(\a_0)= \H^1(\Omega_0)\times\C$, we set
\begin{gather}
\label{J1}
(\J\e^1 f)(\x)= 
\begin{cases}
  \eps^{-1/2}f_1(\Phi\e(x_1)),&\x=(x_1,x_2)\in S\e,\\[1ex]
  \eps^{-1/2}f_1(0) + \dfrac1{h\e} 
  \left( (b\e)^{-1}f_2 -  \eps^{-1/2}f_1(0) \right)x_2,&
  \x=(x_1,x_2)\in  P\e,\\[1ex]
  (b\e)^{-1}f_2,&\x\in R\e,
\end{cases}
\end{gather}
where $\Phi\e \colon\R\to \R$ is a continuous and piecewise linear function
given by
\begin{gather*}
  \Phi\e(x)
  =\begin{cases}
    x,& |x| \ge \dfrac \eps2,\\[1ex]
    \dfrac{2x+d\e}  {2(\eps-d\e)}\eps,& -\dfrac\eps2<x<-\dfrac{d\e}2,\\[2ex]
    \dfrac {2x-d\e}{2(\eps-d\e)}\eps,& \dfrac{d\e}2<x<\dfrac\eps2  ,\\[1ex]
      0,& |x|\leq \dfrac{d\e}2.
    \end{cases}
\end{gather*}

The operator $\wt\J\e^1$ is introduced as a restriction of $\wt\J\e$ onto $\dom(\a\e)$:
\begin{gather}
  \label{J1'}
  \wt\J\e^1 = \wt\J\e\restriction_{\dom(\a\e)}.
\end{gather}
It is easy to see that~\eqref{J1'}  correctly defines a linear operator from
$\H^1(\Omega\e)$ to $\H^1(\Omega_0)\times\C$; here we use the fact that for $u\in \H^1(S\e)$   the function
$v(x_1)\coloneqq \int_{-\eps}^0 u(x_1,x_2)\d x_2$ belongs to $\H^1(\Omega_0)$, namely
\begin{gather}
  \label{H1}
  v'(x_1)=\int_{-\eps}^0 \partial_1u(x_1,x_2)\d x_2,\
  \|v\| _{\L(\Omega_0)}\leq \eps^{1/2}\|  u\|_{\L(S\e)},\
  \|v'\| _{\L(\Omega_0)}\leq \eps^{1/2}\|\partial_1 u\|_{\L(S\e)}.
\end{gather}

\begin{lemma}
\label{lm1}
One has  
\begin{gather}\label{lm1:est}
  \forall f\in\HS^1_0:\quad
  \|\J\e f-\J\e^1 f\|_{\HS\e}
  \leq C_{10}\eps^{\min\{1, \al\}} \|f\|_{\HS_0^1}.
\end{gather}
\end{lemma}

\begin{proof}
  It is easy to see  (cf.~\eqref{J},~\eqref{J1}) that
  \begin{gather}\label{terms}
    \|\J\e f-\J\e^1 f\|_{\HS\e}^2=
    \|\J\e f-\J\e^1 f\|_{\L(Y\e)}^2+
    \|\J\e^1 f\|_{\L(P\e)}^2.
  \end{gather}

  To estimate the first term in~\eqref{terms} we need the following standard Poincar\'e  inequality following from the variational characterisation of the first Dirichlet eigenvalue:   
  \begin{gather}
    \label{Poincare}
    \forall f\in \H_0^1(\mathcal{I})\colon\quad
    \|f\|^2_{\L(\mathcal{I})}\leq \pi^{-2} |\mathcal{I}|^2 \|f'\|^2_{\L(\mathcal{I})},
  \end{gather} 
  where
  $\mathcal{I}$ is a bounded interval and  $|\mathcal{I}|$ denotes its length.
  Using~\eqref{Poincare} and the fact that the function $f-f\circ \Phi\e$ belongs to $\H_0^1(-\frac\eps2,\frac\eps2)$    we get
  \begin{gather}
    \label{JJ}
    \|\J\e f-\J\e^1 f\|_{\L(Y\e)}^2=
    \|f_1-f_1\circ \Phi\e\|^2_{\L(-\frac\eps2,\frac\eps2)}
    \leq
    \pi^{-2} 
    \eps^2\|(f_1-f_1\circ \Phi\e)'\|^2_{\L(-\frac\eps2,\frac\eps2)}.
  \end{gather}
  Taking into account that $2d\e\leq\eps$ (see~\eqref{eps:small:1}), we get
  \begin{equation}\label{g:prop}
    1\leq \Phi'\e(x) \le 2\text{\quad as\quad }
    |x|\in \left(\frac{d\e}2,\frac\eps2\right).
  \end{equation}
  Using~\eqref{g:prop}, the fact  that $\Phi\e(x)=0 $ as $|x|\leq \frac{d\e}2$ and the chain rule, we can further estimate~\eqref{JJ} as
  \begin{align}\notag
    \|\J\e f-\J\e^1 f\|_{\L(Y\e)}^2 
    &\leq 
      2\pi^{-2} 
      \eps^2\left(\|f_1'\|^2_{\L(-\frac\eps2,\frac\eps2)}+
      \|(f_1'\circ \Phi\e) \cdot 
      \Phi\e'\|^2_{\L((-\frac\eps2,\frac\eps2)\setminus [-\frac{d\e}2,\frac{d\e}2])}\right)\\
    \notag
    &\leq
      2\pi^{-2} \eps^2
      \left(\|f_1'\|^2_{\L(-\frac\eps2,\frac\eps2)} + 
      4
      \|f_1'\|^2_{\L(-\frac\eps2,\frac\eps2)}\right)\\\label{termY}
    &\leq
      10\pi^{-2}  \eps^2 \|f_1'\|^2_{\L(-\frac\eps2,\frac\eps2)}
      \leq 10\pi^{-2}\eps^2 \|f\|^2_{\HS_0^1}.
  \end{align}
  
  Finally, we estimate the second term in~\eqref{terms}.   
  Recall that $\wt\Omega_0=\Omega_0\cap \left[-\frac12,\frac12\right]$.
 Using \eqref{J1}, \eqref{dhb}  and Lemma~\ref{lemma:sobolev}, we conclude
  \begin{align}
    \nonumber
    \|\J\e^1 f\|_{\L(P\e)}^2
    &\leq
      \max\left\{\eps^{-1}|f_1(0)|^2,\,(b\e)^{-2}|f_2|^2\right\}h\e d\e\\
    \nonumber
    &\leq
      \max\left\{ \eps^{-1}\ell_{\wt\Omega_0}\|f\|^2_{\H^1(\wt\Omega_0)},\,
      (b\e)^{-2}|f_2|^2\right\}h\e d\e
    \\\nonumber
    &\leq \eps^{-1}h\e d\e\max\left\{\ell_{\wt\Omega_0},
      \eps (b\e)^{-2}\right\}\|f\|^2_{\HS^1_0}\\
    \label{termP}
    &\leq  \gamma \eps^{2\al} \ell_{\wt\Omega_0} \|f\|^2_{\HS^1_0}.
  \end{align}
  In the last estimate we use $\beta<1/2$ as then
  $\eps(b_\eps)^{-2} = \eps^{1-2\beta} \le 1< \ell_{\wt\Omega_0}$.
  Combining~\eqref{terms},~\eqref{termY} and~\eqref{termP}, we arrive
  at the desired estimate~\eqref{lm1:est} with the constant
  explicitly given by $C_{10}=\sqrt{10\pi^{-2} + \gamma \ell_{\wt\Omega_0} }$.
\end{proof}

Now we come to the key lemma of this work; the following estimate on the two forms associated with $\A\e$ and $\A_0$:
\begin{lemma}\label{lm2}
  One has    
  \begin{gather}\label{lm2:est}
    \forall f\in \HS_0^2,\,u\in \HS^2\e\colon\quad
    \left|\a\e[\J^1\e f,u]-\a_0[f, \wt\J^{1}\e u] \right|\leq 
    C_{11}\eps^{\min \{\al, 1/2-\be\}}\|f\|_{\HS_0^2}\|u\|_{\HS^1\e}.
  \end{gather} 
\end{lemma}

\begin{proof}
Let $f=(f_1,f_2)\in\dom(\A_0) = \HS_0^2$ and $u\in \dom(\A\e)=\HS^2\e$.
We denote 
\begin{equation*}f\e \coloneqq \J^1\e f,\qquad u_{\eps,1}\coloneqq (\wt \J^1\e u)_1,\qquad u_{\eps,2}\coloneqq (\wt \J^1\e u)_2\end{equation*}
(i.e.\ $u_{\eps,1}\in \H^1(\Omega_0)$, $u_{\eps,2}\in\C$  are the first and the second components of $\wt \J^1\e u\in \H^1(\Omega_0)\times\C$, respectively).
Taking into account that $f\e $ is constant on $R\e$, and $V\e=0$ on $P\e\cup R\e$, one has
\begin{gather*}
  \a\e[\J^1\e f,u]-\a_0[f, \wt \J^1\e u]=I^1\e+I^2\e+I^3\e,
\end{gather*}
where
\begin{align*}
  I^1\e&\coloneqq  \left(\nabla f\e,\nabla u\right)_{\L(S\e)} - 
         \left(f_1',u_{\eps,1}'\right)_{\L(\Omega_0)},& 
  \text{(the part on the strip)}\\[1ex]  
  I^2\e&\coloneqq  \left(\nabla f\e,\nabla u\right)_{\L(P\e)} - 
         \gamma f_1(0)\overline {u_{\eps,1}(0)},&
  \text{(the part on the passage)}\\[1ex]  
  I^3\e&\coloneqq   (V\e f\e, u)_{\L(S\e)} - (V_0 f_1,u_{\eps,1})_{\L(\Omega_0)}&
  \text{(the potential term).}
\end{align*} 
 
\subsubsection*{Estimate of $I\e^1$ (on the strip)}  It is easy to see that
$I^1\e= \bigl((f_1\circ \Phi\e)' - f_1',
  u_{\eps,1}'\bigr)_{\L(-\frac\eps2,\frac\eps2)}$, therefore we have
\begin{gather}\label{I:1:1}
  |I^1\e|
  \leq 
  \| (f_1\circ \Phi\e)' - f_1 '\| _{\L(-\frac\eps2,\frac\eps2)}
  \|u_{\eps,1}'\| _{\L(-\frac\eps2,\frac\eps2)}.
\end{gather}
Using the chain rule and~\eqref{g:prop}, we obtain
\begin{multline}\label{I:1:3}
  \|(f_1\circ \Phi\e)' - f_1'\| _{\L(-\frac\eps2,\frac\eps2)}\\
  \leq 
  \| (f_1'\circ \Phi\e) \cdot \Phi\e'\|_{\L((-\frac\eps2,\frac\eps2)\setminus [-\frac{d\e}2,\frac{d\e}2])} 
  +\| f_1  '\| _{\L(-\frac\eps2,\frac\eps2)}
  \leq 3\| f'_1 \| _{\L(-\frac\eps2,\frac\eps2)}
\end{multline}
Since $f\in\dom(\A_0)$, one has 
\begin{gather}
\label{f'}
f_1'\in\H^1(\Omega_0^+)\text{\quad and\quad }f_1'\in\H^1(\Omega_0^-).
\end{gather}
Then, by virtue of \eqref{Sobolev1} and \eqref{f'} 
we can continue the (square of) estimate~\eqref{I:1:3} as follows:
\begin{align} \notag
  \|(f_1\circ \Phi\e)' - f_1'\|_{\L(-\frac\eps2,\frac\eps2)}^2
  & \leq  \frac{9\eps}2 \left(\|f\|_{\mathsf{L}^\infty(-{\frac\eps2},0)}^2+
  \|f\|_{\mathsf{L}^\infty(0,{\frac\eps2})}^2\right)\\
  &\leq
  \frac{9\eps}2\max\big\{\ell_{\wt\Omega_0^-};\,\ell_{\wt\Omega_0^+}\big\}\left(\| f_1''\|^2_{\L(\Omega_0^- \cup \Omega_0^+)}
  +\| f_1'\|^2_{\L(\Omega_0)}\right). \label{I:1:4} 
\end{align}
Finally, we need the estimates  
\begin{gather}\label{I:1:5}
  \| f_1'\| _{\L(\Omega_0)}\leq \|f\|_{\HS_0^1}\leq \|f\|_{\HS_0^2},
\end{gather}
and
\begin{align}
  \nonumber
  \| f_1''\| _{\L(\Omega_0^-\cup \Omega_0^+)}
  &\le
    \|\widehat \A_0 f_1+f_1\|_{\L(\Omega_0^-\cup \Omega_0^+)}+ \|V_0 f_1 
    + f_1\| _{\L(\Omega_0^-\cup \Omega_0^+)}\\
  &\leq \|f\|_{\HS^2_0}+(\norm{V_0}_{\mathsf{L}^\infty(\Omega_0)}+1)\|f\| _{\HS_0}
    \leq (\norm{V_0}_{\mathsf{L}^\infty(\Omega_0)}+2)\|f\|_{\HS_0^2}\label{I:1:6}
\end{align}
(to derive these  estimates we have used, in particular,~\eqref{scale}).
Inequalities~\eqref{I:1:4}--\eqref{I:1:6} yield
\begin{gather}\label{I:1:7}
  \|(f_1\circ \Phi\e)' - f_1'\| _{\L(-\frac\eps2,\frac\eps2)}
  \leq C_{12}\eps^{1/2} \|f\|_{\HS_0^2}
\end{gather}
with $C_{12}=
\sqrt{\ds\frac92\max\big\{\ell_{\wt\Omega_0^-};\,\ell_{\wt\Omega_0^+}\big\} \left((\norm{V_0}_{\mathsf{L}^\infty(\Omega_0)} + 2)^2+1\right)}$.
Also, using~\eqref{H1}, we get
\begin{gather}\label{I:1:8}
  \|u_{\eps,1}'\| _{\L(-\frac\eps2,\frac\eps2)}\le 
  \|\partial_1u\|_{\L(Y\e)}\leq 
  \|\nabla u\|_{\L(Y\e)}\leq
  \|u\|_{\HS\e^1}.
\end{gather}
Combining~\eqref{I:1:1},~\eqref{I:1:7},~\eqref{I:1:8} we arrive at the estimate
\begin{gather}
\label{I:1:final}
|I^1\e|\leq C_{12}\eps^{1/2} \|f\|_{\HS_0^2} \|u\|_{\HS\e^1}.
\end{gather}

\subsubsection*{Estimate of $I\e^2$ (on the passage)}
Now we come to the key part of the form estimate, where the $\delta$-potential in the form $\a_0$ appears.  We have
\begin{gather}
  \label{I:2:1}
  \Delta f\e=0,\qquad 
  \partial_1 f\e=0,\qquad 
  \partial_2 f\e=\frac1{h\e}
  \bigl( (b\e)^{-1}f_2 - \eps^{-1/2}f_1(0) \bigr)  \text{\quad on\quad }P\e.
\end{gather}
Using~\eqref{I:2:1} and integrating by parts, we obtain
\begin{align}
  \nonumber
  I^2\e&=\int_{\partial P\e}(\partial_{\mathrm{n}} f\e(\x)) \cdot 
         \overline {u(\x)} \d \x 
         - \gamma f_1(0)\overline{u_{\eps,1}(0)}\\
  \nonumber
       &=\partial_2 f\e \cdot \Bigl(\int_{D\e^+} \overline{u(\x)} \d \x
         -\int_{D\e^-} \overline{u(\x)} \d \x
         \Bigr)
         - \gamma f_1(0) \overline{u_{\eps,1}(0)}\\
  \nonumber
       &=\frac1{h\e} 
         \bigl((b\e)^{-1}f_2 -\eps^{-1/2}f_1(0) \bigr)
         \cdot d\e 
         \bigl(\la \overline u \ra_{D\e^+} - \la \overline u \ra_{D\e^-}\bigr) 
         - \gamma f_1(0) \overline{u_{\eps,1}(0)} \\
  \label{I:2:2}
       &=\gamma \bigl(\eps^{1/2}(b\e)^{-1}f_2 - f_1(0) \bigr) \cdot \eps^{1/2}
         \bigl(\la \overline u \ra_{D\e^+} - \la \overline u \ra_{D\e^-}\bigr) 
         - \gamma f_1(0) \overline{u_{\eps,1}(0)},
\end{align}
where $\partial_{\mathrm{n}}$ stands for the normal derivative on  $\partial P\e$, and where we used $d\e/h\e=\eps\gamma$ in the last line.
Taking into account that
\begin{equation*}
  \eps^{1/2}\la u \ra_{D\e^0}=u_{\eps,1}(0)
  \qquad\text{and}\qquad
  b\e\la u \ra_{R\e}=u_{\eps,2},
\end{equation*}
one can rewrite~\eqref{I:2:2} as follows, 
\begin{gather*}
I^2\e=I\e^{2,1}+I\e^{2,2}+I\e^{2,3}+I\e^{2,4},
\end{gather*}
where
\begin{align*} 
  I\e^{2,1}
  &=\gamma \bigl(\eps^{1/2}(b\e)^{-1}f_2 -f_1(0)\bigr)
    \cdot \eps^{1/2} \bigl(\la \overline u \ra_{D\e^+}
    - \la \overline u \ra_{R\e}\bigr),\\
  I\e^{2,2}
  &=\gamma \bigl(\eps^{1/2}(b\e)^{-1}f_2 -f_1(0)\bigr)
    \cdot \eps^{1/2} (b\e)^{-1} \overline u_{\eps,2},\\%
  I\e^{2,3}
  &=-\gamma \eps (b\e)^{-1} f_2 \la \overline u \ra_{D\e^-},\\
  I\e^{2,4}
  &= \gamma f_1(0) \cdot \eps^{1/2}\bigl(\la \overline u \ra_{D\e^-}
    - \la \overline u \ra_{D\e^0}\bigr).
\end{align*}
We first estimate the common first factor in $I\e^{2,1}$ and $I\e^{2,2}$ by
\begin{equation*}
  \abs[\big]{\eps^{1/2}(b\e)^{-1}f_2 -f_1(0)}^2
  \le \frac {2\eps}{b\e^2} \abs{f_2}^2 
  + 2 \ell_{\wt\Omega_0} \norm{f_1}_{\H^1(\Omega_0)}^2
  \le 2 \ell_{\wt\Omega_0} \norm f_{\HS_0^1}^2
\end{equation*}
using~\eqref{Sobolev1} and the fact that
$\eps(b\e)^{-2}=\eps^{1-2\beta}\le 1<\ell_{\wt\Omega_0}$ provided $\beta<1/2$ (as
usual $\eps \le \eps_0<1$). 
In particular, using~\eqref{lm:aux:2} and the fact that the quantity $\eps^\beta \abs{\ln \eps}^{1/2}$
is uniformly bounded as $\eps\in (0,1]$, namely
\begin{gather}\label{epsbeta:est}
\eps^\beta \abs{\ln \eps}^{1/2}\le (2\beta\euler)^{-1/2}\text{ as }\eps\in (0,1],
\end{gather}
we have
\begin{align*}
  \abs{I\e^{2,1}}
  &\le \gamma (2\ell_{\wt\Omega_0})^{1/2} C_3  \eps^{1/2} \abs{\ln \eps}^{1/2} 
    \norm f_{\HS_0^1} \norm{\nabla  u}_{\L(R\e)}\\
  &\le \gamma (2\ell_{\wt\Omega_0})^{1/2} C_3  \eps^\beta \abs{\ln \eps}^{1/2} 
  \cdot \eps^{1/2-\beta} \norm f_{\HS_0^1} \norm{\nabla  u}_{\L(R\e)}\\
  &\le \underbrace{\gamma (2\ell_{\wt\Omega_0})^{1/2} C_3 (2\beta\euler)^{-1/2}}%
    _{=:C_{13}}
  \cdot \eps^{1/2-\beta} \norm f_{\HS_0^1} \norm{\nabla  u}_{\L(R\e)}.
\end{align*}
Moreover,
\begin{align*}
  \abs{I\e^{2,2}}
  &\le \underbrace{\gamma (2\ell_{\wt\Omega_0})^{1/2}}_{=:C_{14}}
    \cdot \eps^{1/2-\beta} 
    \norm f_{\HS_0^1} \norm u_{\L(R\e)}
\end{align*}
using the fact that $\abs{u_{\eps,2}}\le\norm u_{\L(R\e)}$.  For the third term $I\e^{2,3}$ we need the estimate
\begin{align*}
  \abs[\big]{\la u \ra_{D\e^-}}
  &\le  \abs[\big]{\la u \ra_{D\e^-} -\la u \ra_{Y\e} }
  + \abs[\big]{\la u \ra_{Y\e}}\\
  &\le  C_4 \abs{\ln \eps}^{1/2} 
    \norm{\nabla u}_{\L(Y\e)} + \eps^{-1}\norm u_{\L(Y\e)}\\
  &\le C_4 \abs{\ln \eps}^{1/2} \norm{\nabla u}_{\L(Y\e)}+
  \eps^{-1/2}(\ell_{\wt\Omega_0})^{1/2} \| u\|_{\H^1(S\e)}\\
 &\le \left( C_4 \abs{\ln \eps}^{1/2} + 
  \eps^{-1/2}(\ell_{\wt\Omega_0})^{1/2}\right)   \norm u_{\H^1(S\e)},
\end{align*}
following by~\eqref{lm:aux:2'} and~\eqref{aux:5}.  Hence, we have 
\begin{align*}
  \abs{I\e^{2,3}}
  \le \gamma \eps (b\e)^{-1} \abs{f_2} \abs[\big]{\la u \ra_{D\e^-}}
  &\le \gamma 
   \left( C_4 \abs{\ln \eps}^{1/2} + 
   \eps^{-1/2}(\ell_{\wt\Omega_0})^{1/2}\right)
    \eps^{1-\beta} \abs{f_2}\norm u_{\H^1(S\e)}\\
  &\le \underbrace{\gamma 
    \left( C_4  + 
  (\ell_{\wt\Omega_0})^{1/2}\right)}_{=:C_{15}}
    \eps^{1/2-\beta}
    \abs{f_2} \norm u_{\H^1(S\e)}
\end{align*}
(in the last estimate we use the fact that
$\eps \leq \eps_0<1$, whence $\eps|\ln\eps|<1$).
Finally, we have
\begin{align*}
  \abs{I\e^{2,4}}
  &\le \gamma (\ell_{\wt\Omega_0})^{1/2} C_5 \cdot \eps^{1/2} 
  \abs{\ln \eps}^{1/2}
  \norm {f_1}_{\H^1(\Omega_0)} \norm{\nabla u}_{\L(Y\e)}\\
  &\le \underbrace{\gamma (\ell_{\wt\Omega_0})^{1/2} C_5
    (2\beta\euler)^{-1/2}}_{=:C_{16}}
    \cdot \eps^{1/2-\beta} 
  \norm {f_1}_{\H^1(\Omega_0)} \norm{\nabla u}_{\L(Y\e)},
\end{align*}
using~\eqref{lm:aux:1},~\eqref{Sobolev1}, and~\eqref{epsbeta:est}.  As a result, we arrive at
the estimate
\begin{gather}
  \label{I:2:final}
  |I\e^{2}| \leq (C_{13}+C_{14}+C_{15}+C_{16}) \eps^{1/2-\be}\|f\|_{\HS_1}\|u\|_{\HS_1}.
\end{gather}

\subsubsection*{Estimate of $I\e^3$ (the potential term)}
By virtue of~\eqref{J},~\eqref{J*},~\eqref{V},~\eqref{J1'}
one  gets  the equality
\begin{gather*} 
  I^3\e=  (V\e (\J\e^1 - \J\e)f, u)_{\L(S\e)},
\end{gather*}
hence, using Lemma~\ref{lm1}, we obtain the estimate
\begin{gather}\label{I:3:final}
  |I^3\e|
  \leq \norm{V_0}_{\mathsf{L}^\infty(\Omega_0)} C_{10} 
  \eps^{\min\{1,\al\}}\|f\|_{\HS_0^1}\|u\|_{\HS\e};
\end{gather}
note that $\|V\e\|_{\mathsf{L}^\infty(S\e)}=\|V_0\|_{\mathsf{L}^\infty(\Omega_0)}$. 

Finally,
combining~\eqref{I:1:final},~\eqref{I:2:final},~\eqref{I:3:final} and
taking into account~\eqref{scale}, we arrive at the desired
estimate~\eqref{lm2:est}
with
\begin{equation*}
  C_{11} = C_{12} + C_{13}+C_{14}+C_{15}+C_{16}
  + \norm{V_0}_{\mathsf{L}^\infty(\Omega_0)} C_{10}
  \qedhere
\end{equation*}
\end{proof}
It follows from~\eqref{J*},~\eqref{J1'},~\eqref{lm1:est} and~\eqref{lm2:est} that the conditions of Theorem~\ref{thA:1} (applied to the spaces and operators as in~\eqref{HHAA}) hold with
\begin{equation*}
  \delta=\max\{C_{10},C_{11}\}\eps^{\min \{\al, 1/2-\be\}}.
\end{equation*}
Hence, applying Theorem~\ref{thA:1}, we immediately arrive at the desired first estimate in~\eqref{th1:1}; the second follows from $\wt\J\e\Res\e - \Res_0 \wt\J\e=(\Res\e\J\e- \J\e \Res_0 )^*=\mathcal{L}\e^*$ and $\norm{\mathcal{L}\e}=\norm{\mathcal{L}\e^*}$.  In particular Theorem~\ref{th1} is proven with
$C_1=4\max\{C_{10},C_{11}\}$.
\begin{remark}[on the error estimates in Theorem~\ref{th1}]
  \label{rem:constants}
  Tracing the model parameters, we see that the constant $C_1=4\max\{C_{10},C_{11}\}$ depends on upper bounds of the parameters $\norm V_{\mathsf{L}^\infty(\Omega_0)}$, $\ell^{-1/2}$, $\beta^{-1/2}$, $\alpha^{1/2}$ and $\gamma$ entering in our model.
  Note that the (worst) error $\eps^{1/2-\beta}$ appears in
  $I_\eps^{2,2}$ and $I_\eps^{2,3}$.
  All other error terms are of better order, namely
  \begin{equation*}
     I_\eps^1=\mathcal O(\eps^{1/2}), \quad
    I_\eps^{2,1}, I_\eps^{2,4}=\mathcal O((\eps \abs{\ln \eps})^{1/2}), 
    \quad\text{and}\quad
    I_\eps^3=\mathcal O(\eps^{\min\{1,\alpha\}})
  \end{equation*}
  are better; and the $\eps^\alpha$-term comes from $\wt\J\e^1 u$ on
  the second component, i.e.\ from the $\L$-estimate of $u$ on $P\e$,
  see~\eqref{termP}.
\end{remark}
 
\subsection{Proof of Theorem~\ref{th2}}

To prove Theorem~\ref{th2} we will use Theorem~\ref{thA:2}.

\begin{lemma}
  \label{lm3}
  One has
  \begin{gather}\label{lm3:est}
    \forall u\in\dom(\a\e) \colon\quad 
    \norm u^2_{\HS\e}
    \leq \mu_\eps \norm{\wt\J\e u}^2_{\HS_0} 
    +  \nu_\eps \a\e[u,u]
    \intertext{with}
    \nonumber
    \mu_\eps = 1+ C_6 \eps^{2\alpha} \qquad\text{and}\qquad
    \nu_\eps 
    = C_{17} \eps^{2\min\{\alpha,\beta\}}.
  \end{gather}
\end{lemma}
\begin{proof}
  For a bounded domain $D$ we have the following (equivalent version of the) Poincar\'e inequality, namely
  \begin{equation}
    \label{eq:poincare}
    \norm u_{\L(D)}^2
    \le \abs D \abs[\big]{\la u \ra_D}^2 
    + \frac 1{\lambda_2(D)} \norm{\nabla u}_{\L(D)}^2
  \end{equation}
  for $u \in \H^1(D)$, where $\lambda_2(D)$ denotes the second (first non-zero) Neumann eigenvalue of $D$. Applying this inequality with $D=\{x_1\} \times (-\eps,0)$, we obtain
  \begin{gather*}
    \int_{-\eps}^0|u(x_1,x_2)|^2\d x_2\leq 
    \abs[\Big]{\eps^{-1/2}\int_{-\eps}^0u(x_1,x_2)\d x_2}^2 +
    \frac{\eps^2}{\pi^2} \int_{-\eps}^0|\partial_2u(x_1,x_2)|^2\d x_2
  \end{gather*}
  for (almost) all $x_1 \in \Omega_0$.  Integrating this inequality over $\Omega_0$ with respect to $x_1$ and taking into account the definition of $\wt\J\e$ in~\eqref{J'} and the fact that $\be<1$, 
  we arrive at the estimate
  \begin{align}
    \nonumber
    \forall u\in\H^1(S\e) \colon \quad
    \norm u^2_{\L(S\e)}
    &\leq \norm{(\wt\J\e u)_1}^2_{\L(\Omega_0)} +
     \frac {\eps^2}{\pi^2} \norm{\partial_2u}^2_{\L(S\e)}
     \\\label{Poincare1}
     &\leq \norm{(\wt\J\e u)_1}^2_{\L(\Omega_0)} +
     \frac {\eps^{2\be}}{\pi^2} \norm{\nabla u}^2_{\L(S\e)}. 
  \end{align}
    Similarly, applying~\eqref{eq:poincare} with $D=R\e$, we obtain
  \begin{gather}
    \label{Poincare2}
    \norm u^2_{\L(R\e)}
    \leq \abs[\big]{(\wt\J\e u)_2}^2  
    + \frac{\eps^{2\beta}}{\pi^2} \norm{\nabla u}^2_{\L(R\e)}
  \end{gather}
  using also $b_\eps=\eps^\beta$.  Finally, by virtue of~\eqref{lm:aux:3} and~\eqref{Poincare1}, we obtain
  \begin{align}    \label{Poincare5}
    \norm u^2_{\L(P\e)}
    \leq C_6\eps^{2\al}
    \Bigl(\norm[\big]{(\wt\J\e u)_1}^2_{\L(\Omega_0)}
    + \frac {\eps^{2\beta}}{\pi^2} \norm{\nabla u}^2_{\L(S\e)} 
      + \norm{\nabla u}^2_{\L(S_\eps)}
      + \norm{\nabla u}^2_{\L(P_\eps)}
      \Bigr),
  \end{align} 
  where $C_6$ is given in~\eqref{C6}.  Summing
  up~\eqref{Poincare1}--\eqref{Poincare5}, and taking into account
  that $\eps^{2\beta}< 1$ and $C_6 > 1/\pi^2$, we
  arrive at the desired estimate~\eqref{lm3:est} with 
  $C_{17}\coloneqq \pi^{-2}+C_6(\pi^{-2}+1)$.
\end{proof}

It follows from~\eqref{th1:1} and~\eqref{lm3:est} that the conditions of Theorem~\ref{thA:2} (applied to the spaces and operators as in~\eqref{HHAA}) hold with 
\begin{equation*}
  \eta=\wt\eta=4C_1\eps^{\max\{\alpha, 1/2-\beta\}},\quad 
  \mu=1,\quad \wt\mu=\mu_\eps=1+C_6\eps^{2\alpha},\quad 
  \nu=0,\quad \wt\nu=\nu_\eps=C_{17}\eps^{2\min\{\alpha,\beta\}},
\end{equation*} 
hence, applying Theorem~\ref{thA:2} with $\kappa=\wt\kappa=1/2$ and taking into account~\eqref{dd}, we immediately arrive at the desired estimate~\eqref{th2:est} with
\begin{equation}
  \label{eq:c2}
  C_2=\max \bigr\{4C_1\sqrt{2(1+C_6)},2C_{17} \bigl\}.
\end{equation}
The error $\eps^{2\beta}$ not yet appearing in the estimates of Theorem~\ref{th1} comes from the contribution of $u$ on
$R_\eps$ in~\eqref{Poincare2}.

\subsection{Proof of the quasi-unitary equivalence}

In this subsection, we additionally show that the identification
operators are also \emph{quasi-unitarily equivalent}, see
Section~\ref{sec:quasi-uni}).  From this concept, also used
in~\cite{P06,P12}, the convergence of operator functions as in
Proposition~\ref{prp:conv.op.fcts} follows.  Note that we have given a
more explicit (and better) estimate on the spectral convergence in
Theorem~\ref{thA:2} here, although it was already shown
in~\cite{P06,P12}.  We have commented on the differences in
Remark~\ref{rem:why-not-quasi-uni}.
\begin{lemma}
  \label{lm4}
  We have  $\wt \J\e \J\e f=f$ for all $f \in \HS_0$ and
  \begin{gather*}
    \forall u\in\dom(\a\e) \colon\quad 
    \norm {u- \J\e \wt\J\e u}_{\HS\e}
    \leq C_{18}\eps^{\min\{\alpha,\beta\}} \norm u_{\HS_\eps^1} 
  \end{gather*}
\end{lemma}
\begin{proof}
  For a domain $D$, the standard Poincar\'e inequality
  \begin{equation}
    \label{eq:poincare0}
    \norm {u-\la u \ra_D}_{\L(D)}^2
    \le \frac 1{\lambda_2(D)} \norm{\nabla u}_{\L(D)}^2
  \end{equation}
  holds (in fact,~\eqref{eq:poincare} is equivalent to~\eqref{eq:poincare0}). Using~\eqref{eq:poincare0} with $D=\{x_1\}\times (-\eps,0)$ and $D=R_\eps$, and the estimate~\eqref{lm:aux:3}, we obtain  
  \begin{align*}
    \norm {u- \J\e \wt\J\e u}^2_{\HS\e}
    &=\int_{\Omega_0} \norm{u(x_1,\cdot)- 
      \la u(x_1,\cdot) \ra_{(-\eps,0)}}_{\L((-\eps,0))}^2 \d x_1 
      + \norm u_{\L(P_\eps)}^2+\norm{u-\la u \ra_{R_\eps}}_{\L(R_\eps)}^2\\
    &\le \frac {\eps^2}{\pi^2}
      \int_{\Omega_0} \norm{\partial_2 u(x_1,\cdot)}_{\L((-\eps,0))}^2 \d x_1 
      \\ &\hspace*{0.1\textwidth}
           + C_6\eps^{2\al}
           \Bigl(\|u\|^2_{\L(S\e)} +\norm{\nabla u}_{\L(S\e)}^2 
      + \|\nabla u\|_{\L(P\e)}^2\Bigr)
      + \frac {\eps^{2\beta}}{\pi^2} \norm{\nabla u}_{\L(R_\eps)}^2\\
    &\le C^2_{18} \eps^{2\min\{\alpha,\beta\}}\norm{u}^2_{\HS^1\e},
  \end{align*}
  with $C_{18}\coloneqq (C_6+\pi^{-2})^{1/2}$ (for the last step note
  that $\beta<1$ and $C_6 > 1/\pi^2$). The lemma is proven.
\end{proof}

\section{Countably many $\delta$-interactions}
\label{sec5}
 
The obtained results can be easily extended to 
Schr\"odinger operators with countably many $\delta$-interactions. 
Namely, let  $\Omega_0=(\ell_-,\ell_+)$ and let $Z \subset \Omega_0$ be an at most countable set satisfying
\begin{gather}
  \label{d:countable}
  \inf_{z,z'\in Z, z \ne z'}|z-z'|>0 .
\end{gather}
Note that if $Z$ is infinite, then~\eqref{d:countable} is possible only for $|\Omega_0|=\infty$.
Let $V_0\in\mathsf{L}^\infty(\Omega_0)$ and $(\gamma_z)_{z \in Z}$ be a family of positive numbers
such that 
\begin{gather}\label{gamma:countable}
  \sup_{z\in Z}\gamma_z<\infty.
\end{gather}
In $\L(\Omega_0)$ we consider the operator $\widehat \A_Z$
defined by
the operation 
\begin{equation*}
  -\dfrac{\d^2}{\d x^2}+V_0
  \quad\text{on}\quad \Omega_0\setminus Z,
\end{equation*}
Neumann conditions at $\ell_-$ (provided $\ell_->-\infty$) and 
$\ell_+$ (provided $\ell_+<\infty$), and $\delta$-coupling with strength $\gamma_z$ at $z\in Z$.
The case 
\begin{gather}\label{KP+}
Z=\Z,\quad \gamma_z=\gamma,\quad V_0 = 0   
\end{gather}
corresponds to the Kronig-Penney model. 
To approximate  $\widehat\A_Z$, we consider the domain 
\begin{gather}
  \label{Omegae}
  \Omega\e
  =\intr(\overline{S\e\cup
    \Bigr(\bigcup_{z\in Z}( P_{z,\eps} \cup R_{z,\eps})
    \Bigr)},
\end{gather}
consisting of the straight strip $S\e=\Omega_0\times (-\eps,0)$, and
the family of ``rooms'' $R_{z,\eps}$ and ``passages'' $P_{z,\eps}$ given
by
\begin{equation*}
  R_{z,\eps}
  =\Bigl(z-\frac{b\e}2, z + \frac{b\e}2 \Bigr)
  \times(h\e,h\e+b\e),\quad
  \ds  P_{z,\eps}=
  \Bigl(z- \frac{d_{z,\eps}}2,z+ \frac{d_{z,\eps}}2\Bigr)
  \times (0,h\e).
\end{equation*}
Here $d\e=\gamma_z\eps^{\al+1}$, $h\e=\eps^\al$, $b\e=\eps^\be$ with $\al>0$ and $0<\be<\frac12$.
We choose $\eps$ to be small enough, such that the bottom part (respectively, the top part) of $\partial P_{z,\eps}$ is contained in the top part of  $\partial S\e$ (respectively, the bottom part of  $\partial R_{z,\eps}$), and moreover the neighbouring rooms are disjoint; this can be achieved due to~\eqref{d:countable}--\eqref{gamma:countable}.  As before, 
\begin{gather}\label{Ae}
  \A\e=-\Delta_{\Omega\e}+V\e, 
\end{gather}
where $\Delta_{\Omega\e}$ is the Neumann Laplacian in $\Omega\e$, and the potential $V\e$ is defined as in~\eqref{V}.

As in the case of a single $\delta$-interaction one has the estimate
\begin{gather}
  \label{th2:count}
  \wt d_\Hausdorff\left(\sigma(\A\e),\sigma(\widehat\A_Z)\cup\{0\}\right)
  \leq C\eps^{\min \{\al, 1/2-\be, 2\be\}},
\end{gather}
where $C>0$ is a constant independent of $\eps$.
The proof of~\eqref{th2:count} is similar to the proof of Theorem~\ref{th2}.  It relies on the abstract Theorems~\ref{thA:1} and~\ref{thA:2} being applied to 
\begin{equation*}
  \HS\e\coloneqq \L(\Omega\e),\quad 
  \A\e \text{ as in~\eqref{Ae},}\quad 
  \HS_0\coloneqq \L(\Omega_0)\oplus \ell^2(Z),
  \quad
  \A_0\coloneqq \widehat \A_Z\oplus 0_{\ell^2(Z)}
\end{equation*}
and appropriately modified operators $\J\e$, $\wt\J\e$, $\J\e^1$, $\wt\J\e^1$.

One of the byproducts of this result is the tool for constructing 
periodic Neumann waveguides with spectral gaps.
Namely, it is known~\cite[Sec.~III.2.3]{AGHH05} that the spectrum of the Kronig-Penney operator~\eqref{KP}  has  infinitely many gaps provided $\ga\not= 0$.
Then, using the  estimate~\eqref{th2:count}, we conclude that for any $m\in\N$  the Neumann Laplacian on the periodic domain $\Omega\e$~\eqref{Omegae} (cf.~\eqref{KP+}) has at least $m$ gaps provided $\eps$ sufficiently small enough. 
 
As another byproduct of our careful calculations of the constants, we
may also allow that $Z=Z_\eps$ depends on $\eps$ in such a way that  $\inf_{z,z' \in Z_\eps, z \ne z'} \abs{z-z'}=2\ell_\eps>0$ is still positive as in~\eqref{d:countable}, but with $\ell_\eps \to 0$.  
Moreover, we may allow that $\gamma_z=\gamma_{z,\eps}$ depends on $\eps$ such that $\gamma_{z,\eps} \to  \infty$ as $\eps \to 0$.  If, for example, $\ell_\eps =\eps^\tau$ and $\gamma_{z,\eps}=\gamma_{z,1}\eps^{-\omega}$, then we can show a spectral estimate as in~\eqref{th2:count} remains valid, namely we have $\wt d_\Hausdorff(\sigma(\A\e),\sigma(\widehat\A_{Z_\eps})\cup\{0\})
 \to 0$ as $\eps \to 0$ provided $\tau>0$ and $\omega>0$ are sufficiently small.  Such results are useful when approximating other point interactions by $\delta$-interactions of the form $\widehat \A_{Z_\eps}$, as e.g.\ done for certain self-adjoint vertex conditions on a metric graph, see~\cite{EP13} and the references therein.

\section*{Acknowledgements}

The work of the first author is partly supported by the Czech Science Foundation (GA\v{C}R) through the project 21-07129S. This research was started when the first author was a postdoctoral researcher in Graz University of Technology; he gratefully acknowledges financial support of the Austrian Science Fund (FWF) through the project M~2310-N32.

\end{document}